\documentclass[11pt]{amsart}
\usepackage{amssymb}
\usepackage{amsmath}
\usepackage{amsthm}
\usepackage{empheq}
\usepackage{tikz}
\usepackage{float}
\usetikzlibrary{calc}
\usetikzlibrary{decorations.pathreplacing}
\usetikzlibrary{patterns}
\usepackage{bm}
\usepackage{comment}
\usepackage{mathrsfs}
\usepackage{hyperref}
\usepackage{tikzscale}
\usepackage{enumitem}
\usepackage[T1]{fontenc}
\usepackage[utf8]{inputenc}
\usepackage{todonotes}
\usepackage{mathtools}
\mathtoolsset{showonlyrefs}
\allowdisplaybreaks

\newtheorem{theorem}{Theorem}[section]
\newtheorem{proposition}[theorem]{Proposition}
\newtheorem{lemma}[theorem]{Lemma}

\newtheorem{corollary}[theorem]{Corollary}

\theoremstyle{definition}
\newtheorem{definition}[theorem]{Definition}
\newtheorem{example}[theorem]{Example}
\newtheorem{remark}[theorem]{Remark}

\numberwithin{equation}{section}

\everymath{\displaystyle}

\DeclareMathOperator{\Osc}{Osc}

\newcommand*\mR{\mathbb{R}}
\newcommand*\Rd{\mathbb{R}^d}
\newcommand*\mP{\mathbb{P}}
\newcommand*\mE{\mathbb{E}}
\newcommand*\diam{{\rm diam}}

\newcommand*\arccot{{\rm arccot}}
\newcommand*\wtn{\widetilde{\eta}}
\newcommand*\eps{\varepsilon}

\newcommand*\semicol{\, ; \,}

\let\svthefootnote\thefootnote
\newcommand\freefootnote[1]{%
  \let\thefootnote\relax%
  \footnotetext{#1}%
  \let\thefootnote\svthefootnote%
}
\newcommand*\unD{\underline{D}}
\usepackage[left = 2cm, right = 2cm, top = 2.5cm, bottom = 2.5cm]{geometry}
\begin{document}

 %35-XX 			Partial differential equations
 %35A15   	Variational methods
 %45-XX 	Functional analysis
 %  	46E35   	Sobolev spaces and other spaces of "smooth'' functions, embedding theorems, trace theorems
  % 	31C05   	Harmonic, subharmonic, superharmonic functions
  %47A55   	Perturbation theory [See also 47H14, 58J37, 70H09, 81Q15]
 %60J35   	Transition functions, generators and resolvents
 %47D08   	Schrodinger and Feynman-Kac semigroups
 %60J57	Multiplicative functionals
 %	35R11  	Fractional partial differential equations
 % 60B10: Convergence of probability measures
 %47D06: One-parameter semigroups and linear evolution equations
 %47D07: Markov semigroups and applications to diffusion processes
 %31B35 Connections of harmonic functions with differential equations in higher dimensions
%35K05 Heat equation
%35S05 Pseudodifferential operators as generalizations of partial differential operators [See also 32W25, 47G30, 47L80, 58J40]
%35S10 Initial value problems for PDEs with pseudodifferential operators
%35S15 Boundary value problems for PDEs with pseudodifferential operators
%35S16 Initial-boundary value problems for PDEs with pseudodifferential operators
%60J50 Boundary theory for Markov processes


\title[Caloric functions and regularity for the fractional Laplacian]{Caloric functions and boundary regularity for the fractional Laplacian in Lipschitz open sets}

\author[G. Armstrong]{Gavin Armstrong}
\email{gavin.k.armstrong@gmail.com}
\author[K. Bogdan]{Krzysztof Bogdan}
\address{Faculty of Pure and Applied Mathematics,
Wroc\l aw University of Science and Technology,
Wyb. Wyspia\'nskiego 27, 50-370 Wroc\l aw, Poland}
\email{krzysztof.bogdan@pwr.edu.pl}
\author[A. Rutkowski]{Artur Rutkowski}
\address{Faculty of Pure and Applied Mathematics,
Wroc\l aw University of Science and Technology,
Wyb. Wyspia\'nskiego 27, 50-370 Wroc\l aw, Poland}
\email{artur.rutkowski@pwr.edu.pl}

\begin{abstract}
     We give Martin representation of nonnegative functions caloric with respect to the fractional Laplacian in Lipschitz open sets.  
The caloric functions are defined in terms of the mean value property for the space-time isotropic $\alpha$-stable L\'evy process. 
To derive the representation, we first establish the existence of the parabolic Martin kernel. This involves proving new boundary regularity results for both the fractional heat equation and the fractional Poisson equation with Dirichlet exterior conditions. Specifically, we demonstrate that the ratio of the solution and the Green function is H\"older continuous up to the boundary.
\end{abstract}
\freefootnote{GA was partially supported by the National Science Center of Poland (NCN) under grant 2016/23/B/ST1/01665. KB was partially supported through the DFG-NCN Beethoven Classic 3 programme, contract no. 2018/31/G/ST1/02252 (National Science Center, Poland) and SCHI-419/11–1 (DFG, Germany). AR~was partially supported by the National Science Center of Poland (NCN) under grant 2019/35/N/ST1/04450.}
%\addtocounter{footnote}{-1}\let\thefootnote\svthefootnote
\maketitle

\section{Introduction}\label{s.I}

Let $0 < \alpha < 2$ and $d \geq 2$. For $u\in C^2_b(\mR^d)$, define
\begin{align}\label{eq:fracLapl}
	(-\Delta)^{\alpha/2} u(x) := \lim\limits_{\eps\to 0^+} \int_{|x-y|>\eps} (u(x) - u(y))\nu(x,y)\, dy,\quad x\in \mR^d,
\end{align}
where $\nu(x,y) = c_{d,\alpha}|x-y|^{-d-\alpha}$, and denote $\Delta^{\alpha/2} := -(-\Delta)^{\alpha/2}$. 
Let $D \subset \mathbb{R}^{d}$ be a nonempty bounded open Lipschitz set with localization radius $r_0\in (0,\infty)$ and Lipschitz constant $\lambda\in (0,\infty)$.  One of our goals is to investigate the 
%integral representations and 
structure of nonnegative solutions to the initial-boundary value problem for the fractional heat equation:
\begin{align}\label{eq:FHE}
	\begin{cases}
	\partial_tu(t,x) = \Delta^{\alpha/2}u(t,x),\quad &t\in (0,T),\ x\in D,\\
	u(t,x) = g(t,x),\quad &t\in (0,T),\ x\in D^c,\\
	u(0,x) = u_0(x),\quad &x\in D.
	\end{cases}
\end{align}
Solutions to \eqref{eq:FHE} are called \textit{caloric functions}. They are defined in terms of the mean value property for the space-time $\alpha$-stable L\'evy process; we refer to Section~\ref{sec:caloric} for details and connections with the classical notion of solution to \eqref{eq:FHE}. As shown by Bogdan \cite{MR1704245} (see also Abatangelo \cite{MR3393247} and Bogdan, Kulczycki, and Kwa\'snicki \cite{MR2365478}), nonnegative \textit{harmonic functions} for the fractional Laplacian on $D$ can be decomposed into a \textit{regular} part, which can be recovered from the exterior values, and a \textit{singular} part, vanishing outside of $D$ and represented as an integral with respect to a finite measure on $\partial D$ of the (elliptic) Martin kernel for $D$ and the fractional Laplacian. Our ultimate goal, which we complete in Section~\ref{sec:repr}, is to give a counterpart of this decomposition for caloric functions.
In particular, in Theorem~\ref{th:repr}, we show that nonnegative caloric functions with $u_0=g=0$ can be expressed as integrals with respect to the \textit{parabolic Martin kernel}. To obtain the representation, we prove several new boundary regularity results for the fractional Laplacian in Lipschitz sets, which are a significant focus of this paper. Needless to say, the results point out directions of development for other nonlocal operators and various classes of open sets.

Singular caloric functions were recently represented by Chan, G\'omez-Castro, and V\'azquez \cite{MR4462819} for domains more regular than Lipschitz, such as $C^{1,1}$ domains. While the authors of \cite{MR4462819} address more general operators than our Dirichlet, or \textit{restricted}, fractional Laplacian, they do so by assuming that the (elliptic) Green function exhibits uniform power-type decay at the boundary. Since for Lipschitz open sets, the behavior of the Dirichlet Green function of the fractional Laplacian is more nuanced (see Jakubowski \cite{MR1991120}), the results of \cite{MR4462819} are not applicable in our setting.
Another difference between \cite{MR4462819} and our work is 
that we do not require any specific regularity or integrability conditions for caloric functions, except for assuming nonnegativity and finiteness of integrals in the mean value property. 
%However, this approach is limited to nonnegative functions.
%This comes at the expense of considering only nonnegative functions. 
Furthermore, in our representation, the \textit{boundary data} may be a measure; for example a Dirac delta represents represents a fixed  parabolic Martin kernel.  Furthermore, in Theorem~\ref{th:noinit}, we demonstrate that even without a prescribed initial condition, $u(\eps,\cdot)$ converges to a measure on $D$ as $\eps\to 0^+$. This measure finitely integrates the function $x\mapsto \mP^x(\tau_D>1)$ on $D$ (see below), similar to the condition used in \cite{MR4462819}.

 In our development, we utilize some basic probabilistic potential theory; see, e.g., Sato \cite{MR1739520}.  Let $X = ( X_{t} )_{t \geq 0}$ be the isotropic $\alpha$-stable L\'{e}vy process in $\mathbb{R}^{d}$. For $x \in \mathbb{R}^{d}$, we denote by $\mP^{x}$ and $\mE^{x}$ the probability and the expectation of the process starting from $x$, and $\mP:=\mP^{0}$, $\mE:=\mE^{0}$. We then consider
\begin{equation} \label{fet}
	\tau_{D} := \inf \{ s > 0 \ : \ X_{s} \notin D\},
\end{equation}
the first exit time of the process $X$ from $D$, and the survival probability:
$$\mP^x(\tau_D > t) = \int_D p_t^D(x,y)\, dy,$$
where $p_t^D$ is the \textit{Dirichlet heat kernel} of $\Delta^{\alpha/2}$ in $D$ (for details see Section~\ref{sec:prelim}).
Furthermore, let $G_D$ be the (elliptic) Green function of $\Delta^{\alpha/2}$ in $D$. We fix arbitrary $t_0\in (0,\infty)$ and $x_0\in D$, reference time and point. 

There are several reasonable ways to define the parabolic Martin kernel in Lipschitz open sets. The general idea is to normalize $p_t^D$  by constructing a ratio 
that converges to a nontrivial limit at the boundary of $D$. 
Each of the following expressions will be called a parabolic Martin kernel:
	\begin{align}
  \label{eq:Yaglom}
		&\eta_{t, Q}(x) := \lim_{D\ni y \to Q} \frac{p_{t}^{D}(x,y)}{\mP^{y} ( \tau_{D} > 1 )},\\
  \label{eq:nxo}      &\eta^{x_0}_{t,Q}(x) :=\lim\limits_{D\ni y \to Q}\frac{p_t^D(x,y)}{G_D(x_0,y)},\\
  \label{eq:paraMK}
		&\widetilde{\eta}_{t,Q}(x) := \lim\limits_{D\ni y\to Q} \frac{p_{t}^D(x,y)}{p_{t_0}^D(x_0,y)}.
	\end{align}
 Here, $t>0$, $x\in D$, and $Q\in\partial D$.
 We recall that the heat kernel plays the role of the Green function for the heat equation, see, e.g., Doob \cite{MR1814344}, Watson \cite{MR2907452}, or Bogdan and Hansen \cite[Subsection 9.4]{hansen2023positive}. This might indicate that $\widetilde{\eta}$ is the canonical parabolic Martin kernel, however $\eta$ and $\eta^{x_0}$ offer a more explicit description of the boundary behavior of $p_t^D$ and are more convenient to handle via the existing elliptic theory. If $D$ is $C^{1,1}$, then one can also normalize $p^D_t$ by using $\delta_D(y)^{\alpha/2}$ with
 \begin{align*}
     \delta_D(y):= \inf \{|x-y|: x\in \partial D\},
 \end{align*}
 see Chen, Kim, and Song \cite{MR2677618}; see also
 \cite{MR4462819}. 
 The next result 
 %We note that the existence of the parabolic Martin kernel(s), stated below, 
 may be considered as a consequence and a follow-up of the approximate factorization \eqref{factorization} of $p_t^D$ by Bogdan, Grzywny, and Ryznar \cite{MR2722789}. 
\begin{theorem} \label{Main1}
Recall that $D\subset \Rd$ is open, bounded, and Lipschitz with localization radius $r_0$, Lipschitz constant $\lambda$, and reference point $x_0$ and time $t_0$. Then, the limits in  \eqref{eq:Yaglom}, \eqref{eq:nxo}, and \eqref{eq:paraMK} exist for all 
$t>0$, $x\in D$, and $Q\in\partial D$. Furthermore, they are finite, strictly positive, continuous in $t$ and $x$, and 
	\begin{align}
		\eta_{1, Q}(x) & \approx  \mP^{x} ( \tau_{D} > 1 ), \quad x\in D,\label{Estimation} \\
		\eta_{t + s, Q}(x) & =  \int_{D} \eta_{t, Q}(z) p_{s}^{D}(z,x)\, dz,\quad 0<s,t<\infty,\quad x\in D. \label{Entrance}
	\end{align}
\end{theorem}
The formula \eqref{Estimation} is a sample of more general estimates for $\eta$, which we give in Corollary~\ref{cor:estimates} below.
The proofs of Theorem~~\ref{Main1} and other results of this section are given later on. 
Here we note that the mere existence of a Martin-type kernel is a deep \textit{boundary regularity}\footnote{Here and below, the term signals relative regularity, i.e., continuity or even H\"older continuity of \textit{ratios} at the boundary.} result. In the elliptic setting, for $G_D$, it is usually proved using the boundary Harnack principle. For solutions of parabolic equations like \eqref{eq:FHE}, we may utilize the elliptic results after expressing the numerators and denominators in 
 \eqref{eq:Yaglom}, \eqref{eq:nxo}, and \eqref{eq:paraMK} as Green potentials. This is precisely our approach---it was used before by Bogdan, Palmowski, and Wang \cite{bogdan2018} for Lipschitz cones at the vertex. We further remark that an early version of proof of Theorem~\ref{Main1} for \eqref{eq:Yaglom} has appeared in the PhD thesis of the first-named author \cite{Gavin}.

To obtain the representation of nonnegative caloric functions, we refine Theorem~\ref{Main1} to ensure a uniform rate of convergence in \eqref{eq:Yaglom}. To this end, we extend the spatial domain of
the functions in \eqref{eq:Yaglom}, \eqref{eq:nxo}, \eqref{eq:paraMK}, by additionally defining, for $t>0$, $x\in D$, and $y\in D$,
\begin{equation}\label{e.eeta}
		\eta^{x_0}_{t,y}(x) :=\frac{p_t^D(x,y)}{G_D(x_0,y)},\quad \eta_{t,y}(x) := \frac{p_t^D(x,y)}{\mP^y(\tau_D>1)},\quad \widetilde{\eta}_{t,y}(x) := \frac{p_t^D(x,y)}{p_{t_0}^D(x_0,y)}.
	\end{equation} 

	\begin{theorem}\label{th:modulus}
	Recall that $D\subset \Rd$ is open, bounded, and Lipschitz with localization radius $r_0$, Lipschitz constant $\lambda$, and reference point $x_0$ and time $t_0$. Fix $r_1\in (0,\infty)$ and $0<T_1<T_2<\infty$. For $x\in D$ and $t\in[T_1,T_2]$, $\eta$, $\eta^{x_0}$,  and $\widetilde{\eta}$ 
	are H\"older continuous in $y$ on $\overline{D}$,  $\overline{D}$, and $\overline{D}\setminus B(x_0,r_1)$, respectively. The H\"older exponents and constants depend only on $d,\alpha,\unD,T_1,T_2$ (for $\eta^{x_0}$ also on $x_0,r_1$; for $\widetilde{\eta}$ also on $t_0,x_0$).
\end{theorem}
%The uniform estimates in 
Here and below, we say constants depend on $\unD$ if  they depend only on $r_0,\lambda$, and an upper bound for $\diam (D)$.
Theorem~\ref{th:modulus} yields the following boundary regularity for the semigroup $$P_t^Df(x):=\int_D p_t^D(x,y)f(y)dy.$$
%corresponding to $p_t^D$. 
\begin{corollary}\label{cor:semigroup}
%Recall that $D\subset \Rd$ is open, bounded, and Lipschitz with localization radius $r_0$, Lipschitz constant $\lambda$, and reference point $x_0$ and time $t_0$.
Fix $r_1\in (0,\infty)$. Let $u_0\in L^1(D)$, $0<T_1<T_2<\infty$, and $t\in [T_1,T_2]$. Then, the functions
    \begin{align*}
        \frac{P_t^D u_0(y)}{G_D(x_0,y)},\quad \frac{P_t^D u_0(y)}{\mP^y(\tau_D>1)},\quad \frac{P_t^Du_0(y)}{p_{t_0}^D(x_0,y)}
    \end{align*}
    are H\"older continuous in $y$ on $\overline{D}\setminus B(x_0,r_1)$, $\overline{D}$, and $\overline{D}$ respectively. The H\"older exponents and constants depend only on $d,\alpha,\unD,T_1,T_2$ (and $t_0,x_0$, $r_1$, where relevant). 
\end{corollary}

Theorem~\ref{th:modulus} and Corollary~\ref{cor:semigroup} can be viewed as analogues of the boundary regularity result for $C^{1,1}$ open sets by Fern\'andez-Real and Ros-Oton \cite[Theorem~1.1 (b)]{MR3462074}, see also \cite{MR3626038}. However, such regularity results for nonlocal equations are quite scarce for Lipschitz and less regular domains. That is, much is known about harmonic functions \cite{MR1438304, MR2365478, MR1991120}, but the first result for the Poisson equation ($\Delta^{\alpha/2}u = -f$) appeared only recently in the work of Borthagaray and Nochetto \cite{MR4530901}, who proved optimal Besov regularity of solutions. For regularity results in $C^{1,\gamma}$ domains with $\gamma\in(0,1)$, see, e.g., Abels and Grubb \cite{MR4578319} or Dong and Ryu \cite{2023arXiv230905193D} and the references therein.

Incidentally, our proof of Theorem~\ref{th:modulus} unveils the following integral estimate for the Green function.
\begin{theorem}\label{lem:lpmodulus}
Recall that $D\subset \Rd$ is open, bounded, and Lipschitz with localization radius $r_0$, Lipschitz constant $\lambda$, and reference point $x_0$.
	Let $r>0$. There exists $p_0 = p_0(d,\alpha,\unD,r)>1$ and constants $C\in (0,\infty)$ and $\sigma\in (0,1]$ depending only on $d,\alpha,\unD,p,r$, such that for all $p\in[1,p_0)$,
	\begin{align*}
		\bigg\|\frac{G_D(y,\cdot)}{G_D(x_0,y)} - \frac{G_D(y',\cdot)}{G_D(x_0,y')}\bigg\|_{L^p(D)} \leq C|y-y'|^{\sigma},\quad y,y'\in \overline{D}\setminus B(x_0,r).
	\end{align*}
\end{theorem}
Recall that Green potentials $v(x) = G_Df(x):=\int_D G_D(x,y)f(y)dy$ solve the Dirichlet problem for the Poisson equation:
\begin{align*}
    \begin{cases}(-\Delta)^{\alpha/2} v(x) = f(x),\quad &x\in D,\\
    \hspace{41pt} v(x) = 0,\quad &x\in D^c,
    \end{cases}
\end{align*} see \cite{MR1825645}. Theorem~\ref{lem:lpmodulus} yields a boundary, or relative, H\"older estimate, as follows. 
\begin{corollary}\label{cor:gdmod}
    Let $p>p_0/(p_0-1)$ and let $f\in L^p(D)$. Then, $G_D f(y)/G_D(x_0,y)$ is H\"older continuous in $\in D\setminus B(x_0,r)$ with H\"older constant and exponent depending only on $d,\alpha,\unD,p,r$ and $\|f\|_{L^p(D)}$.
\end{corollary}
A similar result for $C^{1,1}$ domains was obtained by Ros-Oton and Serra \cite{MR3168912} with explicit and sharp H\"older exponents. Our regularity results are far from being sharp in terms of $p_0$ and $\sigma$, but this is to be expected for Lipschitz sets---some insight about precise boundary behavior can be gained from the results on cones \cite{MR2075671,deblassie1990,MR2213639} or numerical considerations \cite{MR4667264}, but we do not pursue this point here.

 Let us add a few general comments. The mean-value property for fractional caloric functions is important for our development. It was considered before, e.g., by Chen and Kumagai \cite{MR2008600}. Here we focus on the mean-value property in cylinders, which seems adequate for the initial-exterior problem \eqref{eq:FHE}. The advantage of the approach is that from the Ikeda--Watanabe formula we obtain a semi-explicit formula for the Poisson kernel. We also have the following stochastic interpretation: if $u$ satisfies the mean-value property $(0,T)\times D$, then $u(t,x)$ can be recovered from the space-time isotropic $\alpha$-stable process $s\mapsto (t-s,X_s+x)$, which starts from $(t,x)$ at time $s=0$,
 % \textit{backward} in time () 
 by computing the expectation of $u(t-s,X_s+x)$ at the place of the first exit of the process from $(0,T)\times D$. The exit can occur when $x+X_s$ leaves $D$ before time $t$---in which case the exterior conditions affect the expectation---or when the time coordinate $t-s$ reaches $0$---then the initial condition comes into play. Singular caloric functions start to appear once we assume that the mean-value property is satisfied only on $(0,T)\times U$ for all open (relatively compact sets) $U\subset\subset D$. We refer to the book of Freidlin \cite[Theorem~2.3]{MR833742} for a counterpart of this theory for local operators.

 With a view toward applications in probability, we note that the existence of the limit \eqref{eq:Yaglom} indicates how the isotropic $\alpha$-stable process in $D$, conditioned on surviving at least time 1, behaves near the boundary of $D$. More precisely, it implies the existence of a “Yaglom limit”,
 %-like probability measure on our set, 
 see Theorem~\ref{Main2} below. Thanks to \eqref{Entrance}, $\eta_{t,Q}(y)$ may be understood as the \textit{entrance law} for the killed process from $Q$ into $D$, see Blumenthal \cite{MR1138461}. This was used in \cite{HAAS20124054, 2018arXiv180408393K} to describe the behavior of the process started from a point on the boundary, e.g., the apex of a cone. Furthermore, the boundary behavior of the heat kernel yields a measure which represents the probability distribution of a rescaled process conditioned on non-extinction.

 Let us now present an outline of the proofs and methods in this paper. In order to prove Theorem~\ref{Main1} we obtain an explicit representation of the survival probability as a Green potential and we show that it behaves like $G_D(x_0,\cdot)$ at the boundary. Then we \textit{approximate} $p_t^D$ by Green potentials and obtain the limit in \eqref{eq:Yaglom} with the help of Prokhorov theorem. To this end, we utilize 
 %Notably, we do not use any strong boundary results, only 
 the uniform integrability of  ratios of Green functions. The proof of Theorem~\ref{lem:lpmodulus} consists in splitting the integral into one region where the boundary Harnack principle can be applied, and another region where we use a technical interior regularity argument adapted to possible singularities of the Green function. In order to prove Theorem~\ref{th:modulus}, we \textit{represent} $p_t^D$ as a Green potential and we apply Theorem~\ref{lem:lpmodulus}. We make use of the spectral theory to show that $p_t^D$ has  regularity necessary for the proof; some ideas here were inspired by \cite{MR4462819}. The boundary measure in the representation of singular caloric functions is obtained from an approximating sequence constructed via the so-called lateral Poisson kernel. Our construction is quite different than the one in \cite{MR4462819}, in particular it does not use the inhomogeneous fractional heat equation.

The structure of the rest of the paper is as follows. Section~\ref{sec:prelim} contains basic definitions and facts. In Section~\ref{sec:Yaglom}, we prove Theorem~\ref{Main1} and its consequences. In Section~\ref{sec:modulus}, we prove Theorems~\ref{lem:lpmodulus} and \ref{th:modulus}. In Section~\ref{sec:caloric}, we introduce the caloric functions and the parabolic Poisson kernel and study their properties. Then in Section~\ref{sec:repr}, we discuss the representation of nonnegative parabolic functions in Lipschitz cylinders.

\section{Preliminaries}\label{sec:prelim}
We assume throughout that the considered sets, measures, and functions are Borel.
For nonnegative functions $f$ and $g$, we write $f(x)\lesssim g(x)$, $x\in A$, if there is a number $C\in (0,\infty)$, referred to as \textit{constant}, such that $f(x)\leq C g(x)$, $x\in A$. We write $C=C(d,\alpha,\ldots)$ if $C$ is a \textit{constant} depending only on $d,\alpha,\ldots$, that is, $C$ may be considered as a function of the parameters $d,\alpha,\ldots$, but not  of $x\in A$. We say that $f$ and $g$ are \textit{comparable} and write $f\approx g$ if $f\lesssim g$ and $g\lesssim f$ (this notation was used in Section~\ref{s.I}). We often use $:=$ and occasionally employ \textit{cursive} for definitions.
\subsection{Geometry} \label{sec:geom}
Let $B(x,r) := \{y\in \mR^d: |y-x| < r\}$. Recall that $D$ is a Lipschitz open set with constant $\lambda\in (0,\infty)$ and localization radius $r_0\in (0,\infty)$. This means that for every $Q\in \partial D$ there is a rigid motion $R_Q$ and a Lipschitz function $f_Q\colon\mR^{d-1}\to \mR$ with Lipschitz constant $\lambda$, such that $R_Q(Q) = 0$ and $ D\cap B(Q,r_0)=R_Q^{-1}(B(0,r_0)\cap \{y_d > f_Q(y_1,\ldots,y_{d-1})\})$.
%for which $R_Q(\partial D\cap B(Q,r_0))$ is a graph of  and
%\begin{itemize}
%\item 
%\item 
%\end{itemize} 
For $r>0$, we let 
\begin{align}\label{eq:Dn}D_{r} := \{x\in D: \delta_D(x) > 1/r\}.\end{align}
Let $\kappa = 1/(4\sqrt{1+\lambda^2})$. Of course, $\kappa<1$.
%We fix $x_1\in D$ such that $x_1\ne x_0$. 
For $y\in \overline{D}$ and $r>0$, we define
\begin{align}\label{eq:ardef}
	\mathcal{A}_r(y) :=\begin{cases} \{A\in D: B(A,\kappa r)\subseteq D\cap B(y,r)\},\quad &r\leq r_0/2,\\
    \{x_0\},\quad &r>r_0/2.\end{cases}
\end{align}
\begin{lemma}
If $D$ is Lipschitz, then $\mathcal{A}_r(y)$ is nonempty for every $r>0$ and $y\in \overline{D}$.
\end{lemma}
\begin{proof}
    Obviously, it suffices to consider $r\leq r_0/2$. For $y\in \partial D$ the statement is true even with $\kappa$ replaced by $2\kappa = 1/(2\sqrt{1+\lambda^2})$. Indeed, if we consider the \textit{interior right-circular} cone with angle $\arccot(\lambda)$
    and vertex at $y$,
%(which exists because $D$ is Lipschitz), 
then the point $A\in D$ on the axis of the cone such that $|A-y|=r$ satisfies $B(A,r/(2\sqrt{1+\lambda^2})) \subseteq D\cap B(y,r)$. If $y\in D$ and $y\notin \mathcal{A}_r(y)$, then there is $Q\in \partial D$ with $|y-Q| = \delta_D(y) < \kappa r$ and $A\in D$ with
    \begin{align*}
        B(A,r/(4\sqrt{1+\lambda^2})) \subseteq D\cap B(Q,r/2)\subseteq D\cap B(y,r).
    \end{align*}
\end{proof}
Thus, by definition (see, e.g., \cite{MR2722789}), $D$ is $\kappa$-fat at each scale $r\in(0,r_0/2)$. We will denote by $A_r(y)$ an arbitrary point in $\mathcal{A}_r(y)$. The actual choice is unimportant in the sense  that if $A_1,A_2\in \mathcal{A}_r(y)$ and $u\ge 0$ is harmonic in $B(A_1,\kappa r)$ and $B(A_2,\kappa r)$---see Definition~\ref{def:harm} below---then we have the comparability $C^{-1}u(A_1)\le u(A_2)\le C u(A_1)$, where $C=C(d,\alpha)$; see the Harnack inequality in \cite[Lemma~1]{MR1703823}, see also \cite[Lemma 4.4]{MR1825645}.

For $x,y\in D$, let $r_{x,y} := |x-y|\vee \delta_D(x)\vee\delta_D(y)$. Let 
%fix another reference point $x_2\in D$ different from $x_0$ and $x_1$. Let 
$\mathcal{A}_{x,y} := \{x_0\}$ if $r_{x,y}>r_0/32$, and otherwise let 
	\begin{align*}\mathcal{A}_{x,y} := \{A\in D: B(A,\kappa r_{x,y}) \subset D\cap B(x,3r_{x,y})\cap B(y,3r_{x,y})\}.
	\end{align*}
Then, $\mathcal A_{x,y}$ is nonempty, \cite{MR1991120}. We denote by $A_{x,y}$ any point in $\mathcal{A}_{x,y}$. The actual choice is unimportant in the sense that under suitable assumptions on functions $u\ge 0$, there exists $C = C(d,\alpha,\unD)$ such that for all $A_1,A_2\in \mathcal{A}_{x,y}$,
%and harmonic $u$ 
$C^{-1}u(A_1)\leq u(A_2)\leq C u(A_1)$. See Remark~\ref{r.c}, following \eqref{eq:GDest}.

% \begin{remark}\label{rem:R0}
%     The fixed reference points $x_0,x_1,x_2$ can be taken in a way that their distance from the boundary and mutual distances are greater than $R_0 = R_0(r_0,\lambda)>0$.
% \end{remark}
%Since $D$ is Lipschitz, the set $\mathcal{A}_r(y)$ is nonempty for every $y\in \overline{D}$, provided $0<r\leq r_0/2$.

%with constant $\lambda$, it is also $\kappa$-fat, which means that there exists $R\in (0,\infty)$ such that 
\subsection{Potential theory}
As stated in the introduction, we denote by $(X_{t}, \mP^{x})$ the standard rotation invariant $\alpha$-stable L\'{e}vy process in $\mathbb{R}^{d}$. The process is determined by the jump measure with density function
\begin{equation} \label{JumpMeasure}
\nu(y) = \frac{2^{\alpha} \Gamma ( (d + \alpha) / 2 )}{\pi^{d/2} | \Gamma ( -\alpha / 2 ) |} |y|^{-d -\alpha} =: c_{d,\alpha}|y|^{-d -\alpha}, \quad y \in \mathbb{R}^{d}.
\end{equation}
It is a process with independent and stationary increments and characteristic function $\mE^{x} e^{i \langle \xi, X_{t} - x \rangle} = e^{-t | \xi |^{\alpha}}$, $t>0$, $x,\xi \in \mathbb{R}^{d}$.
It is strong Markov with the following time-homogeneous transition probability
\begin{equation*}
P_{t}(x,A) := \int_{A} p_{t}(x,y)\, dy, \quad t > 0, \ x \in \mathbb{R}^{d}, \ A \subseteq \mathbb{R}^{d}.
\end{equation*}
Here $p_{t}(x,y) := p_{t} (x - y)$ and $p_{t}$ is the smooth real-valued function on $\mathbb{R}^{d}$ with the Fourier transform:
\begin{equation}
\int_{\mathbb{R}^{d}} p_{t}(x) e^{i \langle x, \xi \rangle} \, dx = e^{-t |\xi|^{\alpha}}, \quad \xi \in \mathbb{R}^{d}. \label{alpha}
\end{equation}
The associated semigroup of operators acts on, e.g., $u\in C_0(\mR^d)$ as follows:
\begin{equation*}
P_t u(x) := \int_{\mR^d} u(y) p_t(x,y) \, dy, \quad x\in\mR^d,\ t\geq 0.
\end{equation*}
We have the following scaling property as a consequence of \eqref{alpha}:
\begin{equation}
p_{t}(x) = t^{-d/\alpha} p_{1} ( t^{-1/\alpha} x ), \quad x \in \mathbb{R}^{d}, \ t > 0. \label{scaling}
\end{equation}
Furthermore, there exists a constant $c$ such that
\begin{equation}
c^{-1}\bigg( t^{-d/\alpha} \wedge \frac{t}{|x|^{d + \alpha}} \bigg) \leq p_{t}(x) \leq c \bigg( t^{-d/\alpha} \wedge \frac{t}{|x|^{d + \alpha}} \bigg), \quad x \in \mathbb{R}^{d}, \ t > 0, \label{eq:HKest}
\end{equation}
see, e.g., \cite{MR119247, MR2013738}.
Thus, in short,
\begin{equation}
p_{t}(x) \approx t^{-d/\alpha} \wedge \frac{t}{|x|^{d + \alpha}},  \quad x \in \mathbb{R}^{d}, \ t > 0. \label{DensityApprox}
\end{equation}

 Recall that $\tau_D$ is the first exit time from $D$ defined in \eqref{fet}. If $D$ is bounded, then $\tau_{D} < \infty$ almost surely, see, e.g., Pruitt~\cite{pruitt1981}. The Dirichlet heat kernel $p_{t}^{D}(x,y)$ of $D$ is defined by the Hunt's formula:
\begin{equation}\label{eq:dhkdef}
p_{t}^{D} (x,y) = p_{t}(x,y) - \mE^{x} \big[p_{t - \tau_{D}} ( X_{\tau_{D}}, y)\semicol \tau_{D} < t  \big],
\end{equation}
where $x, y \in \mathbb{R}^{d}$ and $t > 0$. Here, as usual,
\begin{equation}
\mE^{x} \big[p_{t - \tau_{D}} ( X_{\tau_{D}}, y)\semicol \tau_{D} < t \big] := \int_{\{ \tau_{D} < t \}} p_{t - \tau_{D}} ( X_{\tau_{D}}, y )\, d\mP^{x}. \nonumber
\end{equation}
Since $D$ is Lipschitz, it satisfies the exterior cone condition. Therefore, $\mP^{x} ( \tau_{D} = 0) = 1$ for all $x \in D^{c}$ by Blumenthal's zero-one law. In particular $p_{t}^{D}(x,y) = 0$ when $x$ or $y$ are outside of $D$.
For bounded or nonnegative functions $f$ we define
\begin{equation*}
P_{t}^{D} f (x) := \mE^{x} \big[f ( X_{t})\semicol \tau_D>t \big] = \int_{\mathbb{R}^{d}} f(y) p_{t}^{D} (x,y)\, dy,
\end{equation*}
see  \cite[Section~2]{MR1329992}. We also note that
\begin{equation}
0 \leq p_{t}^{D} (x,y) = p_{t}^{D} (y,x) \leq p_{t}(y - x) \nonumber
\end{equation}
and $p_{t}$ satisfies the Chapman--Kolmogorov equations:
\begin{equation*}
\int p_{s}^{D}(x,y) p_{t}^{D}(y,z)\, dy = p_{t + s}^{D} (x,z), \quad s, t > 0, \ x, z \in \mathbb{R}^{d},
\end{equation*}
see  \cite{MR2602155, MR2677618}.
The following scaling property follows from \eqref{scaling},
\begin{equation}
p_{t}^{D}(x,y) = t^{-d/\alpha} p_{1}^{t^{-1/\alpha} D} ( t^{-1/\alpha} x, t^{-1/\alpha} y ), \quad x, y \in \mathbb{R}^{d}, \ t > 0. \label{HeatScaling}
\end{equation}
By \cite[Theorem~1]{MR2722789}, for every $T>0$ we have the \textit{approximate factorization}:
\begin{equation}
p_{t}^{D} (x,y) \approx \mP^{x} (\tau_{D} > t )  p_{t}(x,y)\mP^{y} ( \tau_{D} > t ) , \quad x,y \in D,\ t\in (0,T). \label{factorization}
\end{equation}
If $D$ is (open, bounded, and)  $C^{1,1}$, then the estimate takes on a more explicit form \cite{MR2677618}:
\begin{equation}\label{eq:CKS}
p_t^D(x,y) \approx \bigg(1\wedge \frac{\delta_D(x)^{\alpha/2}}{\sqrt{t}}\bigg)p_t(x,y)\bigg(1\wedge \frac{\delta_D(y)^{\alpha/2}}{\sqrt{t}}\bigg),\quad x,y \in D,\ t\in (0,T).
\end{equation}
We also recall the large time estimates. Let $\lambda_1 = \lambda_1(D) >0$ be the first eigenvalue and $\varphi_1$ the first eigenfunction of the Dirichlet fractional Laplacian on $D$, see Section~\ref{sec:spect} below for more details. By the intrinsic ultracontractivity due to Kulczycki \cite{MR1643611}, for every $T>0$ we have
\begin{equation}\label{eq:largetimesgen}
	p_t^D(x,y) \approx e^{-\lambda_1 t}\varphi_1(x)\varphi_1(y),\quad x,y\in D,\ t\in (T,\infty).
\end{equation}
If $D$ is (open, bounded, and)  $C^{1,1}$, then we even have 
\begin{equation}\label{eq:largetimes}
p_t^D(x,y) \approx e^{-\lambda_1 t}\delta_D(x)^{\alpha/2}\delta_D(y)^{\alpha/2},\quad x,y\in D,\ t\in (T,\infty),
\end{equation}
see \cite[Theorem 1.1 (ii)]{MR2677618}.
We define the \textit{killing intensity} of $X$ on $D$ as
\begin{equation*}
\kappa_{D}(z):= \int_{D^{c}} \nu(z - y)\, dy,\quad z\in D.
\end{equation*}
 By \cite[Theorem 31.5]{MR1739520}, $\Delta^{\alpha/2}$ coincides with the generator of $X_t$ for the class $C^2_c(\mR^d)$ of real-valued twice continuously differentiable functions with compact support in $\mR^d$.

The \textit{Green function} of $D$ is given by the formula:
\begin{equation*}
G_{D}(x,y) := \int_{0}^{\infty} p_{t}^{D} (x,y)\, dt, \quad x, y \in \mathbb{R}^{d}.
\end{equation*}
In particular, $G_D(x,y) = 0$ if either $x \in D^{c}$ or $y\in D^c$. We note that $G_D$ is finite for all $x\neq y$ and by \eqref{eq:dhkdef}, $G_D(x,y)\leq G_{\mR^d}(x,y)= c|x-y|^{\alpha-d}$. For further reference, we recall the Green function estimates of Jakubowski
	\cite[Theorem~1]{MR1991120}: If we let $$\Phi(x) := G_D(x_0,x)\wedge 1,$$ then there exists $C(d,\alpha,\unD)> 0$ such that
 \begin{align}\label{eq:GDest}
     C^{-1}|x-y|^{\alpha-d}\frac{\Phi(x)\Phi(y)}{\Phi(A_{x,y})^2}\leq G_D(x,y)\leq C|x-y|^{\alpha-d}\frac{\Phi(x)\Phi(y)}{\Phi(A_{x,y})^2}\,,\qquad x,y\in D,
 \end{align}
 see Subsection~\ref{sec:geom} for notation and the following remark.  
 \begin{remark}\label{r.c}
 We note that if $A_1,A_2\in \mathcal A_{x,y}$, then $\Phi(A_1)\approx \Phi (A_2)$; see \cite[Lemma~13]{MR1991120}.
     We also note that \cite{MR1991120}  uses an extra reference point $x_1$ to define $A_{x,y}$ for $r_{x,y}\ge r_0/32$, but the resulting values of $\Phi(A_{x,y})$ are trivially comparable in both settings. In particular, \eqref{eq:GDest} remains true in the present (simplified) setting.
%For technical reasons, the reference point $x_0$ for $\Phi$ in \cite{MR1991120} was chosen to be different than the $x_0$ appearing in the definition of $\mathcal{A}_{x,y}$, but the estimate that we cite here remains valid with our definition.
 \end{remark}
 \begin{remark}
     \label{r.por}
It is implicit in \eqref{factorization} and \eqref{eq:largetimesgen} that
$\varphi_1(y)\approx\mP^y(\tau_D>1)$, $y\in D$. Furthermore, by \cite[Theorem~2]{MR2722789}, $\mP^y(\tau_D>1)\approx\mE^y \tau_D$, $y\in D$, and, by \cite[Lemma~17]{MR1991120}, $\mE^y \tau_D\approx \Phi(y)$, $y\in D$. Therefore,
\begin{equation}
    \label{e.por}
\varphi_1(y)\approx\mP^y(\tau_D>1)\approx\mE^y \tau_D\approx \Phi(y),\quad y\in D. 
\end{equation}
In our proofs, we mostly use the survival probability and $\Phi$, but we also refer to results stated in terms of $\varphi_1$ and the expected exit time. 
 \end{remark}
 %Also for each $y \in D$, $G(x,y)$ is $\alpha$-harmonic in $x \in D \backslash \{ y \}$ and regular $\alpha$-harmonic in $D \backslash B(y,r)$, for every $r > 0$. 
We define the Green operator (or Green potential)
\begin{equation*}
( G_{D} f ) (x) := \int_{D} G_{D}(x,y) f(y)\, dy,\quad x\in \mR^d,
\end{equation*}
for integrable or nonnegative functions $f$. For $f\in L^1(D)$, the function $u:=G_Df$ is a distributional solution of $(-\Delta)^{\alpha/2} u = f$ in $D$, see \cite[Proposition~3.13]{MR1825645}.

\begin{definition}\label{def:harm}
Let $u \geq 0$ be a Borel measurable function on $\mathbb{R}^{d}$.
\begin{itemize}
\item We say that $u$ is $\alpha${\it-harmonic} in an open set $V \subseteq \mathbb{R}^{d}$ if for every open (relatively compact) $B \subset\subset D$, 
$$
u(x) = \mE^{x} u ( X_{\tau_{B}} ) < \infty, \quad x \in B.
$$
\item  We say that $u$ is \textit{regular $\alpha$-harmonic} in $D \subset \mathbb{R}^{d}$ if
$$
u(x) = \mE^{x} u ( X_{\tau_{D}} ) < \infty, \quad x \in D.
$$
\item  We say that $u$ is \textit{singular $\alpha$-harmonic} in $D \subset \mathbb{R}^{d}$, if $u$ is $\alpha$-harmonic in $D$ and $u = 0$ on $D^{c}$.
\end{itemize}
\end{definition}

We will often write `harmonic' instead of `$\alpha$-harmonic'. Since $\tau_{B} \leq \tau_{V}$ for $B \subset V$, by the strong Markov property it follows that regular harmonic functions are harmonic. Also by the strong Markov property, $G_D(\cdot, y)$ is harmonic in $D\setminus \{y\}$, see \cite[Theorem 2.5]{MR1329992} or \cite[(2.1)]{MR1750907}.

For $x \in \mathbb{R}^{d}$, the $\mP^{x}$-distribution of $X_{\tau_{D}}$ is called the \textit{$\alpha$-harmonic measure}, denoted by $\omega^{x}_{D}$. This measure is concentrated on $D^{c}$ and for $u$ regular harmonic in $D$,  we have
\begin{equation*}
u(x) = \int_{D^{c}} u(z)\, \omega_{D}^{x} (dz), \quad x \in D.
\end{equation*}

The $\alpha$-harmonic measure of a Lipschitz open set is absolutely continuous with respect to the Lebesgue measure. Its density function is given by the \textit{Poisson kernel}:
\begin{equation}\label{eq:IWbasic}
P_D(x,z) := \int_D G_D(x,y)\nu(y,z)\, dy,\quad x\in D,\, z\in D^c,
\end{equation} see \cite[Lemma~6]{MR1438304}. Therefore, for every regular harmonic $u$ we have the representation
\begin{equation*}
u(x) = \int_{D^{c}} P_{D}(x,z) u(z)\, dz, \quad x \in D.
\end{equation*}
We also recall the Ikeda--Watanabe formula from \cite{ikeda1962}:
\begin{equation}\label{eq:IW}
	\mP^{x} \big[ \tau_{D} \in I, X_{\tau_{D}-} \in A, X_{\tau_{D}} \in B \big] = \int_{I} \int_{B} \int_{A} \nu(y,z)p_{u}^{D} (x,dy) \, dz\, du,
\end{equation}
where $I \subset (0,\infty)$, $A \subset D$, and $B \subset ( \overline{D} )^{c}$. See also \cite[Lemma~1]{MR2345912}, \cite{MR1438304}, \cite[(4.13)]{MR3737628}, or \cite[Theorem~2.4]{MR2255353}.
%The \textit{potential operator} $U_{\alpha}$ of the process $X_{t}$ is given by the Riesz kernel of order $\alpha$. That is, for nonnegative Borel measurable functions $f$ on $\mathbb{R}^{d}$
%\begin{equation}
%( U_{\alpha} f ) (x) = \mE^{x} \int_{0}^{\infty} f ( X_{t} ) dt, \hspace{0.3in} x \in \mathbb{R}^{d}.
%\end{equation}
%Given a measure $\mu$ on $\mathbb{R}^{d}$, we let $U_{\alpha}^{\mu}$ be its \textit{Riesz potential},
%\begin{equation}
%U_{\alpha}^{\mu}(x) = \mathscr{A}_{d,\alpha} \int_{\mathbb{R}^{d}} \frac{d\mu}{|x - y|^{d - \alpha}},
%\end{equation}
%where $x \in \mathbb{R}^{d}$ and $\mathscr{A}_{d, \alpha} = \Gamma ( ( d - \gamma ) / 2 ) / ( 2^{\gamma} \pi^{d/2} | \Gamma ( \gamma / 2 ) | )$.

Recall that $x_{0} \in D$ is an arbitrary but fixed (reference) point. We define the \textit{Martin kernel}, $M_{D}^{x_{0}}(y, Q)$ as follows: for every $Q \in \partial D$ and $y \in D$ we let
\begin{equation}
M_{D}^{x_{0}} (y,Q) = \lim_{D\ni x \to Q} \frac{G_{D}(x,y)}{G_{D}(x, x_{0})}. \label{Martin}
\end{equation}
In \cite[Lemma~6]{MR1704245} it is shown that the Martin kernel exists, the mapping $(y, Q) \mapsto M_{D}^{x_{0}}(y,Q)$ is continuous on $D \times \partial D$, and for every $Q \in \partial D$ the function $M_{D}^{x_{0}}(\cdot, Q)$ is singular $\alpha$-harmonic in $D$.

\subsection{Auxiliary results on $P_t^D$ and its spectral decomposition}\label{sec:spect}
We recall that the operators $P_t^D$ are compact on $L^2(D)$, see, e.g., \cite[Chapter~4]{MR2569321}. Therefore there exist a nondecreasing sequence of nonnegative numbers $\lambda_n$ diverging to infinity and an orthonormal sequence of functions $\varphi_n\in C_0(D)$ such that for every $\phi \in L^2(D)$, we have
\begin{align}\label{eq:ptdspec}
	P_t^D\phi(x) = \sum\limits_{n=1}^\infty e^{-\lambda_n t} \langle \phi,\varphi_n\rangle \varphi_n(x)
\end{align}
and
\begin{align}\label{eq:dhkspec}
	p_t^D(x,y) =  \sum\limits_{n=1}^\infty e^{-\lambda_n t} \varphi_n(x)\varphi_n(y),\quad x,y\in D,\; t>0.
\end{align}
The fractional Weyl bounds \cite{MR107298,MR3390410} read
\begin{align}\label{eq:Weyl}
	\lambda_n \approx n^{\alpha/d}.
\end{align}
Note that $P_t^D\varphi_n = e^{-\lambda_n t} \varphi_n$ pointwise. Therefore,
\begin{align}\label{eq:GDvarphi}
	G_D \varphi_n = \int_0^\infty P_t^D \varphi_n \, dt = \lambda_n^{-1} \varphi_n.
\end{align}
By iterating \eqref{eq:GDvarphi} and using the regularity results for the fractional Laplacian \cite{MR3482695,MR4194536}, we find that $\varphi_n$ are smooth in $D$.  Furthermore, by \cite[Proposition~3.1]{MR3462074}, there exist $C>0$ and $w\geq 1$, such that
\begin{align}\label{eq:eigenbound}
	\|\varphi_n\|_{\infty} \leq C \lambda_n^{w-1},\quad n\in\mathbb{N}.
\end{align}
We say that $\phi$ belongs to $D(L^D)$, the domain of the $L^2$-generator of $P_t^D$, if the following limit exists in~$L^2$:
\begin{align*}
	L^D \phi := \lim\limits_{t\to 0^+} \frac{P_t^D\phi - \phi}{t}.
\end{align*}
Furthermore, if the limit exists for a function $\phi$ and some $x\in D$, we denote it as $L^D \phi(x)$.
\begin{lemma}\label{lem:spec}
	\begin{enumerate}[font=\normalfont]
		\item We have $\varphi_n\in D(L^D)$ and $L^D\varphi_n (x)= -\lambda_n \varphi_n(x)$ for all $x\in D$.
		\item We have \begin{align*}
			A := \{\phi\in L^2(D): \sum\limits_{n=1}^\infty \lambda_n^2 |\langle\phi,\varphi_n\rangle|^2 < \infty \} \subseteq D(L^D),
		\end{align*}
		and for each $\phi\in A$,
		\begin{align*}
			L^D\phi = \sum\limits_{n=1}^\infty \lambda_n \langle \phi,\varphi_n\rangle \varphi_n.
		\end{align*}
		\item For every $y\in D$ and $t>0$, $p_t^D(\cdot,y) \in A$. \item For every $x,y\in D$, we have $L^D_x p_t^D(x,y) = \Delta^{\alpha/2}_x p_t^D(x,y)$.
	\end{enumerate}
\end{lemma}
\begin{proof}
	Statements (1) and (2) follow quite easily from \eqref{eq:ptdspec} and \eqref{eq:Weyl}. In order to prove (3), we first let $m\in \mathbb{N}$. Then, by \eqref{eq:dhkspec} and \eqref{eq:eigenbound},
	\begin{align*}
		|\langle p_t^D(\cdot,y),\varphi_m\rangle| = |e^{-\lambda_m t}\varphi_m(x)| \leq e^{-\lambda_m t} \|\varphi_m\|_{\infty} \leq Ce^{-\lambda_m t} \lambda_m^{w-1}.
	\end{align*}
	Using \eqref{eq:Weyl}, we get (3).
	
	 We now prove (4). Note that $x\mapsto p_t^D(x,y)\in C^2(D)\cap C_c(\mR^d)$. Let $\phi \in C^2_c(B(x,\delta_D(x)/2))$ (extended by 0 to the whole of $\mR^d$) and $g\in C_c(\mR^d)$ be such that $\phi(x) + g(x) = p_t^D(x,y)$ and $g(x) = 0$ on $B(x,\delta_D(x)/4)$. Note that by \eqref{DensityApprox},
	\begin{align}\label{eq:ptt}
		\frac{p_t^D(x,z)}{t} \leq \frac{p_t(x,z)}{t} \lesssim \nu(x,z),
	\end{align} 
	which for $|x-z|>\delta_D(x)/4$ is uniformly bounded.
	Furthermore, since $p_t(x,z)/t \to \nu(x,z)$ for all $x,z\in\mR^d$, $x\ne z$, by \eqref{eq:dhkdef} we find that for fixed $z\in D\setminus \{x\}$,
	\begin{align*}
		\lim\limits_{t\to 0^+} \frac{p_t^D(x,z)}{t} = \nu(x,y)  +  \lim\limits_{t\to 0^+} \frac 1t \mE^x[p_{t-\tau_D}(X_{\tau_D},z)\semicol \tau_D<t].
	\end{align*}
	Since $x$ and $z$ are fixed we have $p_{t-\tau_D}(X_{\tau_D},z) \lesssim t$, so the limit on the right hand side is equal to 0, hence $p_t^D(x,z)/t\to \nu(x,z)$ as well. By this, \eqref{eq:ptt}, and the dominated convergence theorem, we get $\Delta^{\alpha/2}g(x) = L^Dg(x)$.
	
	Let $L$ be the $C_0(\mR^d)$-generator of the semigroup induced by $p_t$. By Sato \cite[Theorem~31.5]{MR1739520}, we have $\Delta^{\alpha/2} \phi(x) = L\phi(x)$. Therefore,
	\begin{align*}
		L^D\phi(x) = \Delta^{\alpha/2}\phi(x) + \lim\limits_{t\to 0^+} \frac{P_t^D\phi(x) - P_t\phi(x)}{t}.
	\end{align*}
	We will show that the last limit exists and is equal to 0. By \eqref{eq:dhkdef}, Fubini--Tonelli, and the fact that $X_{\tau_D}\in D^c$ almost surely,
	\begin{align*}
		\frac{|P_t^D\phi(x) - P_t\phi(x)|}{t} \leq \|\phi\|_{\infty} \frac 1t \mE^x\bigg[\int_{B(x,\delta_D(x)/2)} p_{t-\tau_D}(X_{\tau_D},z)\, dz\semicol  \tau_D<t\bigg]\lesssim \mP^x(\tau_D<t) \mathop{\longrightarrow}\limits^{t\to 0^+} 0.
	\end{align*}
	By collecting the above results we find that
	\begin{align*}
		\Delta^{\alpha/2}_x p_t^D(x,y) = 	\Delta^{\alpha/2}\phi(x) + 	\Delta^{\alpha/2} g(x) = L^D\phi(x) + L^D g(x) = L^D_x p_t^D(x,y),
	\end{align*}
	which ends the proof.
\end{proof}
\begin{corollary}\label{cor:lptd}
	For every $t>0$, $\Delta^{\alpha/2}_x p_t^D$ is bounded in $D\times D$.
\end{corollary}
\begin{proof}
	By Lemma~\ref{lem:spec} and \eqref{eq:eigenbound}, we have 
	\begin{align*}
		|\Delta^{\alpha/2}_x p_t^D(x,y)| = |L^D_x p_t^D(x,y)| = \bigg|\sum\limits_{n=1}^\infty \lambda_n e^{-\lambda_n t} \varphi_n(x)\varphi_n(y)\bigg| \leq \sum\limits_{n=1}^\infty C\lambda_n e^{-\lambda_n t} \lambda_n^{2w-2} \leq C_0<\infty.
	\end{align*}
\end{proof}
%The functions $\varphi_n$ form a basis of eigenfunctions of $L$ on $L^2(D)$ with Dirichlet exterior condition and $\lambda_n$ are the corresponding eigenvalues. In other words, we have (in the pointwise sense)
%\begin{align}\label{eq:eigenproblem}
%	\begin{cases}
%		L\varphi_n(x) = \lambda_n \varphi_n(x),\quad &x\in D,\\
%		\varphi_n(x) = 0,\quad &x\in D^c.
%	\end{cases}
%\end{align}
\begin{lemma}\label{lem:gdpdl}
	Let $\phi\in C_c^\infty(D)$. Then,
	\begin{align*}
		P_t^D L\phi(y) = \sum\limits_{n=1}^\infty e^{-\lambda_n t}\lambda_n \langle \phi,\varphi_n\rangle \varphi_n(y)
		,\quad y\in D.
	\end{align*}
\end{lemma}
\begin{proof}
	Note that $L\phi\in L^2(D)$, hence
	\begin{align*}
		P_t^D L\phi(y) = \sum\limits_{n=1}^\infty e^{-\lambda_n t} \langle \varphi_n,L\phi\rangle \varphi_n(y).
	\end{align*}
	By \eqref{eq:GDvarphi} we have $\varphi_n = G_D[\lambda_n\varphi_n]$. Therefore, by \cite[Proposition~3.13]{MR1825645},
	\begin{align*}
		\langle  \varphi_n,L\phi\rangle = 	\langle  G_D[\lambda_n\varphi_n],L\phi\rangle = \langle \lambda_n\varphi_n,\phi\rangle,
	\end{align*}
	which ends the proof of the lemma.
	\end{proof}
	The following result is a weighted Hausdorff--Young type inequality.
	\begin{lemma}\label{lem:HY}
		There exist $c = c(d,\alpha,\unD)$ and $w\in \mathbb{N}$ such that for any $p\in[2,\infty]$ and $u\in L^p(D)$,
		\begin{align*}
			\|u\|_{L^p(D)} \leq c \bigg(\sum\limits_{n=1}^\infty |\langle u,\varphi_n\rangle|^{p'} \lambda_n^{w-1}\bigg)^{1/p'},
		\end{align*}
		where $p' = p/(p-1)$ is the H\"older conjugate exponent of $p$.
	\end{lemma}
	\begin{proof}
		Let $\phi\in L^2(D)$. By \eqref{eq:eigenbound}, we have $\|\varphi_n\|_{\infty}\leq C\lambda_n^{w-1}$ for some $C>0$ and $w\geq 1$ independent of $n$. Therefore for $x\in D$,
		\begin{align*}
			\|\phi\|_{\infty} \leq \sum\limits_{n=1}^\infty |\langle \phi,\varphi_n\rangle| \|\varphi_n\|_{\infty} \leq C \sum\limits_{n=1}^\infty |\langle \phi,\varphi_n\rangle| \lambda_n^{w-1}.
		\end{align*}
		If we let $\hat{\phi} = (\langle \phi,\varphi_1\rangle,\langle \phi,\varphi_2\rangle,\ldots)$ and denote by $l^p_\lambda$ the space of sequences with the $p$-th powers summable with the weight $(\lambda_1^{w-1},\lambda_2^{w-1},\ldots)$, then the above means that $\hat{\phi}\mapsto \phi$ is bounded from $l^1_\lambda$ to $L^\infty(D)$. By Parseval's identity, this map is also bounded from $l^2$ to $L^2(D)$, hence also from $l^2_\lambda$ to $L^2(D)$. The statement of the lemma follows from the Riesz--Thorin theorem.
	\end{proof}
\section{Yaglom limits in Lipschitz open sets}\label{sec:Yaglom}

In this section we prove Theorem~\ref{Main1}. We first establish the asymptotics of Green potentials at the boundary points of $D$. This extends what is already known about the asymptotics of Green potentials at the vertex of cone \cite[Lemma~3.5]{bogdan2018}; we also propose a different proof.
\begin{lemma} \label{Lemma7}
If $f$ is a measurable function bounded on $D$ and $Q \in \partial D$, then
$$
\lim_{x \to Q} \int_{D} \frac{G_{D}(x,y)}{G_{D} (x, x_{0})} f(y)\, dy = \int_{D} \lim_{x \to Q} \frac{G_{D}(x,y)}{G_{D} (x, x_{0})} f(y)\, dy < \infty, \quad x \in D.
$$
\end{lemma}

\begin{proof}
 Fix two points $x_{1}, x_{2} \in D$ and let 
\begin{equation}
\rho = ( \delta_D(x_1)  \wedge \delta_D(x_2)  \wedge |x_1-x_2| )/3, \label{rho}
\end{equation}
so that $B(x_{1}, \rho), B(x_{2}, \rho) \subset D$ and $B(x_{1}, \rho) \cap B(x_{2}, \rho) = \emptyset$. We know that $M_{D}^{x_{0}}( \cdot, Q )$ given by \eqref{Martin} is regular $\alpha$-harmonic on $B(x_{1}, \rho)$ and $B(x_{2}, \rho)$, and for $x$ sufficiently close to $\partial D$ so is $G_{D} (x,\cdot)$.
% Using the definition of the Martin kernel (\ref{Martin}) and the $\alpha$-harmonicity of both the Martin kernel and the Green function we get
Therefore, for $i=1,2$,
\begin{align*}
\int_{B(x_{i}, \rho)^{c}} \lim_{x \to Q} \frac{G_{D}(x,y)}{G_{D} (x, x_{0})} \,\omega_{B(x_{i}, \rho)}^{x_{i}}(dy) & =  \int_{B(z_{i}, \rho)^{c}} M_{D}^{x_{0}}(y, Q)\, \omega_{B(z_{i}, \rho)}^{z_{i}}(dy) \nonumber \\
& =  M_{D}^{x_{0}}(x_{i}, Q) \nonumber \\
& =  \lim_{x \to Q} \frac{G_{D}(x, x_{i})}{G_{D}(x, x_{0})} \nonumber \\
& =  \lim_{x \to Q} \frac{\int_{B(x_{i}, \rho)^{c}} G_{D}(x,y) \,\omega_{B(x_{i}, \rho)}^{x_{i}}(dy)}{G_{D}(x, x_{0})} \nonumber \\
& =  \lim_{x \to Q} \int_{B(x_{i}, \rho)^{c}} \frac{G_{D}(x,y)}{G_{D} (x, x_{0})}\, \omega_{B(x_{i}, \rho)}^{x_{i}}(dy). \nonumber
\end{align*}
The $\alpha$-harmonic measures $\omega^{x_i}_{B(x_i,\rho)}(dy)$ are absolutely continuous and have radially decreasing densities $g_i$, see, e.g., \cite{MR1704245}. Therefore there exists $C>0$ such that $\omega^{x_i}_{B(x_i,\rho)}(dy) = g_i(y)\, dy \geq C$ for $y\in D\cap(B(x_i,\rho)^c)$. Let $g = g_1 + g_2$. Vitali's theorem \cite[Theorem 16.6 (i) and (iii)]{MR3644418} yields the following $L^1$ convergence:
$$\lim\limits_{x\to Q}\int_D \bigg|\frac{G_D(x,y)}{G_D(x,x_0)}g(y) -  M_{D}^{x_{0}}(y, Q) g(y)\bigg|\, dy = 0.$$
Since $|f| \lesssim C \lesssim g$, the result follows.
\end{proof}
We can also establish the following identity, an analogue of \cite[(3.16)]{bogdan2018}.
\begin{lemma} \label{L38}
For $x \in \mathbb{R}^{d}$, we have
\begin{equation} \label{38}
\mP^{x} ( \tau_{D} > 1 ) = ( G_{D} P_{1}^{D} \kappa_{D} )(x).
\end{equation}
\end{lemma}

\begin{proof}
Let $x \in D$. Since our set $D$ is Lipschitz, from Lemma 6 and the proof of Lemma 17 in \cite{MR1438304},
\begin{align*}
\omega_{D}^{x} ( \partial D) = \mP^{x} ( X_{\tau_{D}} \in \partial D ) &= 0,\\
\mP^{x} ( X_{\tau_{D}-} = X_{\tau_{D}}) &= 0,\\
	\mP^x (X_{\tau_D-} \in D) &= 1.
\end{align*}
By the Ikeda--Watanabe formula \eqref{eq:IW} and the Chapman--Kolmogorov equations we have
\begin{align}
\mP^{x} ( \tau_{D} > 1 ) & =  \mP^{x} \big[ \tau_{D} > 1, \ X_{\tau_{D}-} \in D, \ X_{\tau_{D}} \in D^{c} \big] \nonumber \\
& =  \int_{1}^{\infty} \int_{D^{c}} \int_{D} p_{s}^{D} (x,z) \nu(z - w) \, dz \, dw \, ds \nonumber \\
& =  \int_{\mathbb{R}^{d}} \int_{D^{c}} \int_{0}^{\infty} p_{t + 1}^{D} (x,z) \nu(z - w) \, dt \, dw \, dz \nonumber \\
& =  \int_{\mathbb{R}^{d}} \int_{D^{c}} \int_{0}^{\infty} \int_{D} p_{t}^{D} (x,y) p_{1}^{D} (y,z) \, dy \, \nu(z - w) \, dt \, dw \, dz \nonumber \\
& =  \int_{D} \int_{0}^{\infty} p_{t}^{D} (x,y) \, dt \int_{\mathbb{R}^{d}} p_{1}^{D} (y,z) \int_{D^{c}} \nu(z - w) \, dw \, dz \, dy \nonumber \\
& =  \int_{D} G_{D} (x,y) \int_{\mathbb{R}^{d}} p_{1}^{D} (y,z) \kappa_{D}(z) \, dz \, dy \nonumber \\
& =  \int_{D} G_{D} (x,y) ( P_{1}^{D} \kappa_{D} ) (y) \, dy \nonumber \\
& =  ( G_{D} P_{1}^{D} \kappa_{D} ) (x). \nonumber
\end{align}
For $x \in D^{c}$, both sides of \eqref{38} are equal to 0. This ends the proof.

\end{proof}

We define
$$
C_{1} := \int_{D} \int_{D} M_{D}^{x_{0}} (y, Q) p_{1}^{D}(y,z) \kappa_{D}(z)\, dz\, dy.
$$
Combining the two lemmas above, we obtain the following result.
\begin{lemma} \label{Thm3}
We have $0 < C_{1} < \infty$ and
$
\lim_{x \to Q} \frac{\mP^{x} ( \tau_{D} > 1 )}{G_{D}(x,x_{0})} = C_{1}.
$
\end{lemma}

\begin{proof}
By Lemma~\ref{L38}, $\mP^x(\tau_D>1) = (G_DP_1^D\kappa_D)(x)$. Note that $(P_{1}^{D} \kappa_{D})(y)$ is bounded. Indeed, by \eqref{factorization},
\begin{align}
( P_{1}^{D} \kappa_{D} ) (y) & = \int_{D} p_{1}^{D} (y, z) \kappa_{D} (z) \, dz \nonumber \\
& \approx \mP^{y} (\tau_{D} > 1 ) \int_{D} \mP^{z} (\tau_{D} > 1 )  p_{1}(y,z) \kappa_{D}(z) \, dz,\quad y\in D. \label{ThreeFive}
\end{align}
Since $D$ is bounded, by \eqref{DensityApprox}
\begin{equation}
p_{1}(y, z) \approx 1, \quad y, z \in D. \label{OneVar}
\end{equation}
Hence \eqref{ThreeFive} becomes
\begin{equation} \label{TheIntegral}
( P_{1}^{D} \kappa_{D} ) (y) \approx \mP^{y} (\tau_{D} > 1) \int_{D} \mP^{z} (\tau_{D} > 1) \kappa_{D}(z) \, dz,\quad y\in D.
\end{equation}
Using \eqref{38}, we see that for $x \in \mathbb{R}^{d}$,
\begin{equation}
\int_{D} G_{D}(x,y) ( P_{1}^{D} \kappa_{D} )(y)\, dy = ( G_{D} P_{1}^{D} \kappa_{D} )(x) = \mP^{x} ( \tau_{D} > 1 ) \leq 1. \nonumber
\end{equation}
By \eqref{eq:GDest}, $G_{D}(x,y)$ is strictly positive for all $x, y \in D$. Thus $P_{1}^{D} \kappa_{D}$ has to be finite almost everywhere. Hence the integral in \eqref{TheIntegral} is finite and
$$
( P_{1}^{D} \kappa_{D} ) (y) \approx \mP^{y} (\tau_{D} > 1),
$$
for $y \in D$. In particular, $( P_{1}^{D} \kappa_{D} ) (y)$ is bounded on $D$. By using Lemma~\ref{Lemma7} with $f(y) = ( P_{1}^{D} \kappa_{D} ) (y)$,
\begin{align*}
\lim_{x \to Q} \frac{\mP^{x} ( \tau_{D} > 1 )}{G_{D}(x,x_{0})} & =  \lim_{x \to Q} \frac{( G_{D} P_{1}^{D} \kappa_{D} ) (x)}{G_{D}(x,x_{0})} \nonumber \\
& =  \lim_{x \to Q} \int_{D} \frac{G_{D}(x,y)}{G_{D}(x,x_{0})} ( P_{1}^{D} \kappa_{D} )(y) \, dy \nonumber \\
& = \int_{D} M_{D}^{x_{0}} (y,Q) ( P_{1}^{D} \kappa_{D} )(y) \, dy  =  C_{1} < \infty. \nonumber
\end{align*}
\end{proof}
We are now in a position to prove Theorem~\ref{Main1}.
\begin{proof}[Proof of {\rm Theorem~\ref{Main1}}]
Let us define
\begin{equation}\label{eq:mux}
m_{x} (A) := \frac{\int_{A} p_{1}^{D} (x,y) dy}{\mP^{x} ( \tau_{D} > 1 )}, \quad x \in D, \ A \subseteq \mathbb{R}^{d}.
\end{equation}
First we note that the family $\{ m_{x} :  x \in D \}$ is tight. Indeed, combining the factorization of $p_{1}^{D}(x,y)$ in \eqref{factorization} with the equation \eqref{OneVar}, we get
\begin{equation}
\frac{p_{1}^{D}(x,y)}{\mP^{x} ( \tau_{D} > 1 )} \approx \mP^{y} ( \tau_{D} > 1 ), \qquad x, y \in D.\label{Estim}
\end{equation}
Since the densities of the measures $m_{x}(A)$ are bounded by an integrable function, the tightness follows.

Next we wish to prove that the measures $m_{x}$ converge weakly to a probability measure $m_{Q}$ on $D$ as $x \to Q$. To this end, consider an arbitrary sequence $\{ x_{n} \}$ such that $x_{n} \to Q$. By tightness, there exists a subsequence $\{ x_{n_{k}} \}$ such that $m_{x_{n_{k}}} \implies m_{Q}$ for some probability measure $m_{Q}$, as $k \to \infty$. We will show that this limit is unique.

Let $\phi \in C_{c}^{\infty}(D)$ and $u_{\phi} = (-\Delta)^{\alpha/2} \phi$. For $x \in \mathbb{R}^{d}$, we claim that 
\begin{equation}
( P_{1}^{D} \phi ) (x) = ( G_{D} P_{1}^{D} u_{\phi} ) (x). \label{Statement}
\end{equation}
To show this, we first remark that $u_{\phi}\in C_0(\mR^d)$ and that $( G_{D} u_{\phi} )(x) = \phi(x)$, see \cite[Lemma 5.7]{MR0193671} and \cite[(11)]{MR2365478}. By \eqref{DensityApprox} it follows that
\begin{eqnarray}
( P_{1}^{D} | u_{\phi} | ) (x) = \int_{D} p_{1}^{D}(x,y) | u_{\phi}(y) |\, dy \leq c<\infty. \nonumber
\end{eqnarray}
Therefore, since for a fixed $z\in D^c$ we have $\nu(y,z) \gtrsim 1$ for $y\in D$, by \eqref{eq:IWbasic} we get
\begin{align*}
( G_{D} P_{1}^{D} | u_{\phi} | ) (x) & =  \int_{D} G_{D}(x,y) ( P_{1}^{D} | u_{\phi} | )(y)\, dy \nonumber \\
& \leq  c \int_{D} G_{D}(x,y) \, dy < \infty.
\end{align*}
As a result, we can apply Fubini--Tonelli theorem and establish \eqref{Statement} as follows:
\begin{align}
( G_{D} P_{1}^{D} u_{\phi} )(x) & =  \int_{D} \int_{D} \int_{0}^{\infty} p_{t}^{D}(x,y) p_{1}^{D} (y,z) u_{\phi} (z)\, dt\, dz\, dy \nonumber \\
& =  \int_{D} \int_{0}^{\infty} p_{t + 1}^{D}(x,z) u_{\phi} (z)\, dt\, dz \nonumber \\
& =  \int_{D} \int_0^\infty \int_D p_1^D(x,y)p_t^D(y,z) u_\phi(z) \, dy \, dt \, dz\nonumber \\
%& = & \int_{D} \int_{0}^{\infty} p_{t + 1}^{D}(z,x) u_{\phi} (z) \,dt\, dz \nonumber \\
%& = & \int_{D} \int_{D} \int_{0}^{\infty} p_{t}^{D}(z,y) p_{1}^{D} (y,x) u_{\phi} (z) \,dt \,dz\, dy \nonumber \\
& =  \int_D \int_D \int_0^\infty p_1^D(x,y) p_t^D(y,z) u_\phi(z) \, dt \, dz \, dy\nonumber \\
& =  ( P_{1}^{D} G_{D} u_{\phi} )(x) =  ( P_{1}^{D} \phi )(x). \label{eq:GDLcommute}
\end{align}

Let us denote $m_{x} (\phi) := \int_{D} \phi(y)\, m_{x} (dy)$. Using \eqref{Statement}, Lemma~\ref{Thm3}, and Lemma~\ref{Lemma7}, we get
\begin{align}
\lim_{x \to Q} m_{x}(\phi) & =  \lim_{x \to Q} \frac{( P_{1}^{D} \phi )(x)}{\mP^{x} ( \tau_{D} > 1 )} \nonumber \\
& =  \lim_{x \to Q} \frac{( P_{1}^{D} G_{D} u_{\phi} )(x)}{\mP^{x} ( \tau_{D} > 1 )} \nonumber \\
& =  \lim_{x \to Q} \frac{( G_{D} P_{1}^{D} u_{\phi} )(x)}{G_{D} (x,x_{0})} \frac{G_{D}(x,x_{0})}{\mP^{x} ( \tau_{D} > 1 )} \nonumber \\
& =  \frac{1}{C_{1}} \int_{D} M_{D}^{x_{0}}(y,Q) ( P_{1}^{D} u_{\phi} ) (y)\, dy. \label{limit}
\end{align}
In particular, $m_{Q} (\phi) := \lim_{k \to \infty} m_{x_{n_{k}}}(\phi)$ does not depend on the choice of the subsequence. Thus, by the Portmanteau Theorem, $m_{x} \implies m_{Q}$ as $x \to Q$.

For $t > 1$, we consider $\phi_{t,y}(\cdot) := p_{t-1}^{D}(\cdot, y) \in C_{0} ( \mathbb{R}^{d} )$, see \cite{MR2722789} or \cite[Proposition 1.19]{MR1329992}. Using Chapman--Kolmogorov, we get
\begin{align*}
\eta_{t, Q}(y) & =  \lim_{x \to Q} \frac{p_{t}^{D}(x,y)}{\mP^{x} ( \tau_{D} > 1 )} \nonumber \\
& =  \lim_{x \to Q} \frac{\int_{D} p_{t - 1}^{D} (z,y) p_{1}^{D} (x,z)\, dz}{\mP^{x} ( \tau_{D} > 1 )} \nonumber \\
& =  \lim_{x \to Q} \frac{( P_{1}^{D} p_{t - 1}^{D} ( \cdot ,y) )(x)}{\mP^{x} ( \tau_{D} > 1 )} \nonumber \\
& =  \lim_{x \to Q} m_{x} ( p_{t - 1}^{D} (\cdot, y) ). \nonumber
\end{align*}
By \eqref{limit}, the existence of $\eta_{t, Q}(y)$ for $t > 1$ follows:
\begin{equation}
\eta_{t, Q}(y) = m_{Q} ( p_{t - 1}^{D} (\cdot, y) ). \nonumber
\end{equation}
Note that the threshold $t>1$ is arbitrary, that is, $1$ can be replaced with any $t_0>0$. Indeed, the results of this section can be readily reformulated with $t_0$ in place of $1$, for instance, Lemma~\ref{Thm3} may be strengthened to assert that for every $t_0>0$,
$$\lim\limits_{x\to Q} \frac{\mP^x(\tau_D>t_0)}{G_D(x,x_0)} = \int_D\int_D M_D^{x_0}(y,Q)p_{t_0}^D(y,z)\kappa_D(z)\, dz\, dy.$$
Accordingly, we get the existence of the limit
\begin{equation}\label{eq:restated}\lim\limits_{x\to Q} \frac{\mP^x(\tau_D > 1)}{\mP^x(\tau_D > t_0)}.\end{equation}
We can also reuse the above arguments to get for all $t>t_0$, the existence of
\begin{equation}\label{eq:t0}\lim\limits_{x\to Q} \frac{p^D_t(x,y)}{\mP^x(\tau_D > t_0)}.\end{equation}
Of course, \eqref{eq:restated} and \eqref{eq:t0} give the existence of $\eta_{t,Q}(y)$ for $t>t_0$.

The equation \eqref{Estimation} follows from equation \eqref{Estim}, and the equation \eqref{Entrance} follows from the Chapman--Kolmogorov equations and the dominated convergence theorem (see  \cite[(27)]{MR2722789}):
\begin{equation*}
\eta_{t + s, Q}(y) = \lim_{x \to Q} \int_{D} \frac{p_{t}^{D} (x,z)}{\mP^{x} ( \tau_{D} > 1 )} p_{s}^{D}(z,y)\, dz = \int_{D} \eta_{t,Q}(z) p_{s}^{D}(z,y)\, dz.
\end{equation*}
The fact that $\wtn$ and $\eta^{x_0}$ exist follows from the existence of $\eta$ and from Lemma~\ref{Thm3}.
\end{proof}
\begin{corollary}\label{cor:MYcont}
	The functions $(t,y)\mapsto \eta_{t,Q}(y), \wtn_{t,Q}(y), \eta^{x_0}_{t,Q}(y)$ are continuous on $(0,\infty)\times \overline{D}$.
\end{corollary}
\begin{proof}
	By Theorem~\ref{Main1} and the fact that $p_t^D(x,y)$ and $\mP^y(\tau_D>1)$ are continuous for $(t,y)\in (0,\infty)\times D$, and separated from 0 in sufficiently small neighborhood of any point $(t,y)$, it suffices to verify that for any sequence $((t_n,y_n))\subset (0,\infty)\times D$ such that $(t_n,y_n) \to (t,Q)\in (0,\infty)\times \partial D$, we have 
	\begin{align}\label{eq:tnyn}
\lim_{n\to\infty} \frac{p_{t_n}^D(x,y_n)}{\mP^{y_n}(\tau_D>1)} = \eta_{t,Q}(x).
	\end{align}
	Furthermore, by Theorem~\ref{Main1}, in order to obtain \eqref{eq:tnyn} it suffices to prove that for any $t>0$ there exists a modulus of continuity $\omega$ independent of $y$ such that
	\begin{align}\label{eq:timecont}
	\bigg|\frac{p_{t+\varepsilon}^D(x,y)-p_{t}^D(x,y)}{\mP^y(\tau_D>1)}\bigg| \leq \omega(\varepsilon),\quad \varepsilon>0.
	\end{align}
    By Chapman--Kolmogorov, we have
    \begin{align*}
        \bigg|\frac{p_{t+\varepsilon}^D(x,y)-p_{t}^D(x,y)}{\mP^y(\tau_D>1)}\bigg| &\leq \int_D \frac{|p_t^D(z,y) - p_t^D(x,y)|p_{\varepsilon}^D(x,z)}{\mP^y(\tau_D>1)} \, dz\\
        &=\bigg(\int_{D\setminus B(x,\delta_D(x)/2)} + \int_{B(x,\delta_D(x)/2)}\bigg) \frac{|p_t^D(z,y) - p_t^D(x,y)|p_{\varepsilon}^D(x,z)}{\mP^y(\tau_D>1)} \, dz =: I_1+I_2.
    \end{align*}
    Then by \eqref{factorization},
    \begin{align*}
        I_1 \leq \int_{D\setminus B(x,\delta_D(x)/2)} p_\varepsilon^D(x,z)\, dz \leq \int_{D\setminus B(x,\delta_D(x)/2)} p_\varepsilon(x,z)\, dz \leq \omega(\epsilon).
    \end{align*}
    For $I_2$, we use the gradient bounds of Kulczycki and Ryznar \cite[Theorem~1.1]{MR3767143} and \eqref{factorization}:
    \begin{align*}
        I_2 &\leq \int_{B(x,\delta_D(x)/2)} \frac{|p_t^D(z,y) - p_t^D(x,y)|p_{\varepsilon}^D(x,z)}{\mP^y(\tau_D>1)} \, dz\\
        &\leq \int_{B(x,\delta_D(x)/2)} |x-z| \frac{\|\nabla_x p_t^D(\cdot,y)\|_{L^\infty(B(x,\delta_D(x)/2)}}{\mP^y(\tau_D>1)} p_\varepsilon^D(x,z)\, dz\\
        &\lesssim \int_{B(x,\delta_D(x)/2)} |x-z| \frac{\| p_t^D(\cdot,y)\|_{L^\infty(B(x,\delta_D(x)/2)}}{\mP^y(\tau_D>1)} p_\varepsilon^D(x,z)\, dz\\
        &\lesssim \int_{B(x,\delta_D(x)/2)} |x-z|p_\varepsilon^D(x,z)\, dz \leq \int_{B(x,\delta_D(x)/2)} |x-z|p_\varepsilon(x,z)\, dz \leq \omega(\varepsilon).
    \end{align*}
    Thus, $I_1+I_2 \leq \omega(\varepsilon)$, which ends the proof for $\eta$. For $\wtn$ and $\eta^{x_0}$, we use  Lemma \ref{L38} and \eqref{Estimation}.
\end{proof} 

Here is a rough result about the behavior of $\eta_{s,Q}(x)$ away from the singularity at $(0,Q)$.
 \begin{lemma}\label{lem:nonsingular}
     If $Q\in \partial D$ then $(s,x)\mapsto \eta_{s,Q}(x)$ is locally bounded on $((0,\infty)\times \mR^d)\setminus \{(0,Q)\}$. Furthermore, if $t=0$ or $y\in \partial D$, but $(t,y)\neq (0,Q)$, then
%for every $(t,y)\in ([0,\infty)\times \partial D)\cup (\{0\}\times D)$ other than $(0,Q)$, 
$\eta_{s,Q}(x) \to 0$ as $(s,x) \to (t,y)$.
 \end{lemma}
 \begin{proof}
     By \eqref{factorization} and \eqref{DensityApprox}, we have
     \begin{align*}
         \eta_{s,Q}(x) &= \lim\limits_{D\ni\xi\to Q} \frac{p_s^D(x,\xi)}{\mP^\xi(\tau_D>1)} \lesssim \limsup\limits_{D\ni\xi\to Q} \frac{\mP^\xi(\tau_D>s)}{\mP^\xi(\tau_D>1)}p_s(x,\xi)\mP^x(\tau_D > s)\\
         &\lesssim |x-Q|^{-d-\alpha}\mP^x(\tau_D > s)\limsup\limits_{D\ni\xi\to Q} \frac{s\mP^\xi(\tau_D>s)}{\mP^\xi(\tau_D>1)}.
     \end{align*}
If $|x-Q|\geq \varepsilon$ then $\eta_{s,Q}(x)$ is
%The last three factors are 
bounded---it even converges to $0$ as $s\to 0$---see
%(the rightmost one because of 
Lemma~\ref{lem:survmonot}.
%, where it is also . Furthermo ; 
%under $\limsup$ is small for $s$ small by Lemma~\ref{lem:survmonot}. 
If $s>\varepsilon$ then we use the approximate factorization of $p^D$---and the fact that $\mP^x(\tau_D > s)\to 0$ as $x\to y\in \partial D$. 
 \end{proof}
Let us summarize estimates of $\eta$ that follow from the estimates of the Dirichlet heat kernel.
 \begin{corollary}\label{cor:estimates}
     If $D$ is $C^{1,1}$, then
     \begin{align}\label{eq:C11est}
         \eta_{t,Q}(x) \approx\begin{cases} \frac{1}{\sqrt{t}}\bigg(1\wedge \frac{\delta_D^{\alpha/2}(x)}{\sqrt{t}}\bigg)p_t(x,Q),\quad &t\in (0,1),\ x\in D,\ Q\in \partial D,\\
         e^{-\lambda_1 t} \delta_D(x)^{\alpha/2},\quad &t\in [1,\infty),\ x\in D,\ Q\in \partial D.
         \end{cases}
     \end{align}
     If $D$ is Lipschitz, then 
     \begin{align}\label{eq:Liplarge}\eta_{t,Q}(x) \approx
         e^{-\lambda_1 t} \mP^x(\tau_D > t),\quad t\in [1,\infty),\ x\in D,\ Q\in \partial D,
    \end{align}
and 
   \begin{align}\label{eq:Lipsmalls}
   \eta_{t,Q}(x) \approx
\frac{\mP^x(\tau_D>t)p_t(x,Q)}{\Phi(A_{t^{1/\alpha}}(Q))},
\quad t\in (0,1),\ x\in D,\ Q\in \partial D.
    \end{align} 
Furthermore, there exist $0<\sigma_1\leq \sigma_2<1$ such that
\begin{align}\label{eq:Lipsmall}
         t^{-\sigma_1}\lesssim \frac{\eta_{t,Q}(x)}{\mP^x(\tau_D>t)p_t(x,Q)} \lesssim t^{-\sigma_2},\quad t\in (0,1),\ x\in D,\ Q\in \partial D.
    \end{align}
 \end{corollary}
 \begin{proof}
 The estimate \eqref{eq:C11est} follows from \eqref{eq:CKS} and \eqref{eq:largetimes}. By \cite[Theorem~1.1]{MR1643611} and \eqref{factorization}, $\mP^y(\tau_D>1) \approx \varphi_1(y)$, so \eqref{eq:Liplarge} is a consequence of \eqref{eq:largetimesgen}. It remains to prove \eqref{eq:Lipsmalls} and
 \eqref{eq:Lipsmall}. By \eqref{factorization},
 \begin{align}\label{e.ape}
     \eta_{t,Q}(x)
     \approx      \mP^x(\tau_D>t)p_t(x,Q)
\lim\limits_{y\to Q} \frac{\mP^y(\tau_D>t)}{\mP^y(\tau_D > 1)}.
	\end{align} 
 By \cite[Theorem~2]{MR2722789} and \eqref{e.por}, 
 \begin{align}\label{e.ipp}
    \frac{\mP^y(\tau_D>t)}{\mP^y(\tau_D>1)} \approx  \frac 1 {\Phi(A_{t^{1/\alpha}}(y))}.
 \end{align}
 By geometrical considerations, we can choose points $A_{t^{1/\alpha}}(y)$ converging to a point in $\mathcal A_{t^{1/\alpha}}(Q)$. This proves \eqref{eq:Lipsmalls}.
 By \eqref{e.ape} and Lemma~\ref{lem:survmonot}, we get the upper bound in \eqref{eq:Lipsmall}. The lower bound follows from \eqref{eq:Lipsmalls} and \cite[Lemma~3]{MR1438304} with some $\sigma_1>0$. Of course, we must have $\sigma_1\le \sigma_2$ in \eqref{eq:Lipsmall}.
 \end{proof}
A consequence of Theorem~\ref{Main1} is the Yaglom-type limit, obtained in the thesis of the first author~\cite{Gavin}.
\begin{theorem} \label{Main2}
Suppose that $D$ is a bounded Lipschitz open set such that $0\in \partial D$ and $D\cup\{0\}$ is star-shaped at $0$. If $x \in D$ then for every Borel $A\subseteq\mR^d$,
\begin{equation*}
\lim_{t \to \infty} \mP^{x} \bigg( \frac{X_{t}}{t^{1/\alpha}} \in A \ \bigg| \ \bigg( \frac{X_{s}}{t^{1/\alpha}} \bigg)_{0 \leq s \leq t} \subset D \bigg) = m_{0}(A),
\end{equation*}
where $\mP^x(A_1|A_2) := \mP^x(A_1\cap A_2)/\mP^x(A_2)$ is the conditional probability and $m_{0}(A) := \int_A \eta_{1,{0}}(y)\, dy$.
\end{theorem}
\begin{proof}
Let $x\in D$, $t\geq 1$, and let $A \subset \mathbb{R}^{d}$ be Borel. Then we have
\begin{align*}
\mP^{x} \bigg( \frac{X_{t}}{t^{1/\alpha}} \in A \ \bigg| \ \bigg( \frac{X_{s}}{t^{1/\alpha}} \bigg)_{0 \leq s \leq t} \subset D \bigg) & =  \frac{\mP^{x} ( X_{t} \in t^{1/\alpha}A, \ ( X_{s})_{0 \leq s \leq t} \subset t^{1/\alpha}D )}{\mP^{x} (( X_{s})_{0 \leq s \leq t} \subset t^{1/\alpha}D )} \nonumber \\
& =  \frac{\int_{t^{1/\alpha} A} p_{t}^{t^{1/\alpha} D} (x,y)\, dy}{\int_{t^{1/\alpha}D} p_{t}^{t^{1/\alpha}D} (x,y)\, dy} \nonumber \\
& =  \frac{\int_{t^{1/\alpha} A} t^{-d/\alpha} p_{1}^{D} ( t^{-1/\alpha} x, t^{-1/\alpha} y )\, dy}{\int_{t^{1/\alpha} D}t^{-d/\alpha} p_{1}^{D} ( t^{-1/\alpha} x, t^{-1/\alpha} y )\, dy} \nonumber \\
& =  \frac{\int_{A} p_{1}^{D} ( t^{-1/\alpha} x, y )\, dy}{\int_{D}  p_{1}^{D} ( t^{-1/\alpha} x, y )\, dy}  =  m_{t^{-1/\alpha} x}(A), \nonumber
\end{align*}
where $m_{t^{-1/\alpha}x}$ is the measure defined in \eqref{eq:mux} above (note that $t^{-1/\alpha}x\in D$).
Therefore, by Theorem~\ref{Main1}, this probability approaches $m_{0}(A)$ as $t \to \infty$.
\end{proof}
\section{H\"older regularity}\label{sec:modulus}
This section is devoted to proving Theorems~\ref{lem:lpmodulus} and \ref{th:modulus}. The proof of Theorem~\ref{lem:lpmodulus} uses a mix of the boundary Harnack principle and interior H\"older regularity. Then Theorem~\ref{th:modulus} follows by using the formulas of Section~\ref{sec:Yaglom}, which enable us to relate the heat kernel regularity to the elliptic regularity.

% Let $\eta$ be a standard mollifier. For $\eps>0$ and $x\in \mR^d$ we define $\eex(z) := \eta_\eps(x-z)$.
% \begin{lemma}
% 	For every $x,y\in D$ and $t>0$ we have
% 	\begin{align*}
% 		p_t^D(x,y) = \lim\limits_{\eps\to 0^+} P_t^D\eex(y).
% 	\end{align*}
% 	If $n_0\in \mathbb{N}$, $T_2>T_1>0$ are fixed, $x\in D_{n_0}$ and $t\in [T_1,T_2]$, then the limit is uniform in $y\in D$.
% \end{lemma}
% \begin{proof}
% 	If $\eps < 1/(2n_0)$, then by the fundamental theorem of calculus and \cite[Theorem~1.1]{MR3767143}, 
% 	\begin{align*}
% 		|P_t^D\eex(y) - p_t^D(x,y)| &\leq \int_D |p_t^D(z,y) - p_t^D(x,y)| \eta_\eps(x-z)\, dz\\
% 		&\leq C(d,\alpha,T_1,T_2,n_0) \int_0^1 \int_D p_t^D(x+\theta(z-x),y) |x-z|  \eta_\eps(x-z)\, dz\, d\theta\\
% 		&\leq \widetilde{C}(d,\alpha,T_1,T_2,n_0)\eps,
% 	\end{align*}
% 	which ends the proof.
% \end{proof}
Fix $n_0\geq 2$ such that the reference points $x_0$ belongs to $D_{n_0/2}$.
\begin{lemma}\label{lem:gdquot}
	There exists $p_0=p_0(d,\alpha,\unD)>1$ such that the family $\{(G_D(y,\cdot)/G_D(x_0,y))^p: y\in  D\}$ is uniformly integrable in $D$ for all $p\in[1,p_0)$.
\end{lemma}
\begin{proof} 
	For $y\in D_{n_0}$ we have a crude bound:
	\begin{align*}
		G_D(y,z)/G_D(x_0,y) \leq C(d,\alpha,\unD) |y-z|^{\alpha-d},\quad z\in D.
	\end{align*}
	Considering the functions on the right-hand side, we see that $p_0 = d/(d-\alpha)$ will do. 
	
	From now on assume that $y\in D\setminus D_{n_0}$.  By \eqref{eq:GDest}, there exists $C=C(d,\alpha,\unD)$ such that
	\begin{align*}
		\frac{G_D(y,z)}{G_D(x_0,y)} \leq C \frac{|y-z|^{\alpha-d}}{|x_0-y|^{\alpha-d}} \frac{\Phi(z)\Phi(A_{x_0,y})^2}{\Phi(x_0)\Phi(A_{y,z})^2}.
	\end{align*}
	We immediately get that
	\begin{align*}
		\frac{G_D(y,z)}{G_D(x_0,y)} \leq C' |y-z|^{\alpha-d} \frac{\Phi(z)}{\Phi(A_{y,z})^2}.
	\end{align*}
	By the Carleson estimate \cite[Lemma~13]{MR1991120}, we further find that $\Phi(z)/\Phi(A_{y,z})\leq C(d,\alpha,\unD)$. If we let $U = D_{32/r_0}\cup (D\setminus B(y,r_0/32))$, then it follows that 
	\begin{align*}
		\frac{G_D(y,z)}{G_D(x_0,y)} \leq\begin{cases} C(d,\alpha,\unD) |y-z|^{\alpha-d},\quad &z\in U,\\  C(d,\alpha,\unD)|y-z|^{\alpha-d} \Phi(A_{y,z})^{-1},\quad &z\in D\setminus U.\end{cases}
	\end{align*}
	The definition of $A_{y,z}$ implies that for $z\in D\setminus U$ there exists $Q = Q(z)\in \partial D$ such that $y,z\in B(Q,3r)$ and $B(A_{y,z},\kappa r) \subset D\cap B(Q,6r)$. Using \cite[Lemma~5]{MR1438304}, we find that there exist $C=C(d,\alpha,\unD)$ and $\gamma = \gamma(d,\alpha,\unD)\in (0,\alpha)$ such that
	\begin{align*}
		\Phi(A_{y,z}) \geq C |A_{y,z} - Q(z)|^\gamma \geq C\kappa^\gamma r^\gamma\geq C\kappa^\gamma|y-z|^{\gamma}.
	\end{align*}
	Therefore,
	\begin{align}\label{eq:gdquotbound}
		\frac{G_D(y,z)}{G_D(x_0,y)} \leq C (|y-z|^{\alpha-d}\vee |y-z|^{\alpha-\gamma-d}),\quad y\in D\setminus D_{n_0},\ z\in D,
	\end{align}
	so the statement of the lemma holds for all $p\in [1,d/(d-\alpha+\gamma))$. We can take $p_0 = d/(d-\alpha+\gamma)$.
\end{proof}
The following lemma is a specific Carleson-type estimate.
\begin{lemma}\label{lem:gdauxest}
	Let $0<r<\delta_D(y)$, $|z-y|\geq 2r$, and $|v-y|\leq r$. There exists $C = C(d,\alpha,\unD)$ such that
	\begin{align*}
		G_D(z,v) \leq C G_D(z,y).
	\end{align*}
\end{lemma}
\begin{proof}
	Note that $2|z-v|\geq |z-y|$. By \eqref{eq:GDest}, there is $c = c(d,\alpha,\unD)$ such that 
	\begin{align*}
		G_D(z,v) \leq c \frac{\Phi(z)\Phi(v)}{\Phi(A_{z,v})^2} |z-v|^{\alpha-d} \leq 2^{d-\alpha} c \frac{\Phi(z)\Phi(v)}{\Phi(A_{z,v})^2} |z-y|^{\alpha-d}.
	\end{align*}
%see Lemma~\ref{lem:gdquot} for the notation.
	By elementary calculations, we find that $2r_{z,v} \geq r_{z,y}$. By \cite[Lemma~13]{MR1991120} we therefore get $\Phi(A_{z,v})\geq c(d,\alpha,\unD)\Phi(A_{z,y})$. Furthermore, by \cite[Lemma~4 and 5]{MR1438304} we get $\Phi(v) \leq c(d,\alpha,\unD)\Phi(y)$. This ends the proof.
\end{proof}
The next lemma can be viewed as a more concrete, quantified version of \cite[Lemma~8]{MR2365478}. We give an interior-type H\"older regularity for ratios of Green functions, taking into account the singularity at the diagonal. The structure of the proof follows the boundary regularity approach of \cite[Lemma~16]{MR1438304}, but here the singularity can occur between the boundary and the arguments of the function, see Figure~\ref{fig:oko}.
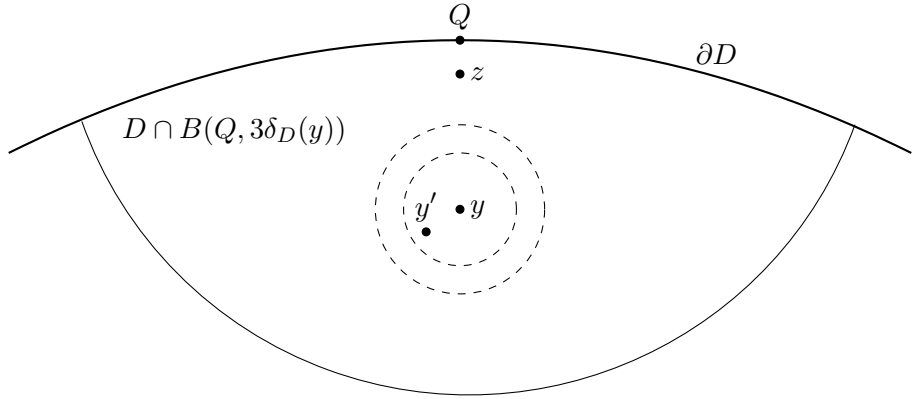
\begin{figure}[H]
    \centering
    \begin{tikzpicture}[scale=1.5]
        \draw[thick, smooth, domain=-4:4, variable=\x] plot({\x},{1-\x*\x/16});
        \filldraw (0,1) circle(1pt) node[above]{$Q$};
        \filldraw (0,0.7) circle(1pt) node[right]{$z$};
        \filldraw (0,-0.5) circle(1pt) node[right]{$y$};
        \filldraw (-0.3,-0.7) circle(1pt) node[above]{$y'$};
        \draw[dashed] (0,-0.5) circle(0.5);
        \draw[dashed] (0,-0.5) circle(0.75);
        \draw (3.5,0.25) arc (-20:-160.5:3.64);
        \draw (2,0.7) node[above right] {$\partial D$};
        \draw (-2,0.2) node {$D\cap B(Q,3\delta_D(y))$};
    \end{tikzpicture}
    \caption{Illustration for Lemma~\ref{lem:gdint}. The boundary Harnack principle cannot be used to estimate increments between $y$ and $y'$ because of the singularity at $z$. Instead we show regularity in the smaller ball using harmonicity in the larger ball.}
    \label{fig:oko}
\end{figure}
\begin{lemma}\label{lem:gdint}
	Let $y\in D$ and $Q\in \partial D$ satisfy $|Q-y| = \delta_D(y)$. Assume that $z\in D\cap B(Q,3\delta_D(y))$ and let $r = |z-y|/4$, so that $\overline{B(y,r)}\subset D$. Then there exist constants $C \geq 1$, $k_0\geq 4$, $\sigma\in (0,1]$, and $\gamma\in (0,\alpha)$, depending only on $d,\alpha,\unD$, such that for every $y'\in B(y,2^{-k_0}r)$ we have
	\begin{align*}
		\bigg|\frac{G_D(y,z)}{G_D(x_0,y)} - \frac{G_D(y',z)}{G_D(x_0,y')}\bigg| \leq C\bigg(\frac{|y-y'|}{r}\bigg)^{\sigma} r^{\alpha-d-\gamma}.
	\end{align*}
\end{lemma}
\begin{proof}
	Note that $B(y,r)\subseteq D$. Let
	\begin{align*}
		&B_k = B(y,(2^{k_0})^{-k}r),\quad k=0,1,\ldots,\\
		&\Pi_k = B_k\setminus B_{k+1},\quad k=0,1,\ldots,\quad \Pi_{-1} = D\setminus B_0,\\
		&u(y) = G_D(y,z),\quad v(y) = G_D(x_0,y).
	\end{align*}
	We will show that there exist $c=c(d,\alpha,\unD)$ and $\zeta = \zeta(d,\alpha,\unD)\in (0,1]$, such that for $k=0,1,\ldots$, 
	\begin{align}\label{eq:oscillation}
		\sup\limits_{B_k} \frac{u}{v} \leq (1 + c\zeta^k)\inf\limits_{B_k} \frac{u}{v}.
	\end{align}
	By virtue of \eqref{eq:gdquotbound}, this implies the statement of the theorem. 
	
	For $-1\leq l < k$ we define
	\begin{align*}
		&u_{k}^l(x) = \mE^x[u(X_{\tau_{B_k}})\, ;\, X_{\tau_{B_k}}\in \Pi_l],\quad v_{k}^l(x) = \mE^x[v(X_{\tau_{B_k}})\, ;\, X_{\tau_{B_k}}\in \Pi_l],\quad x\in \mR^d.	\end{align*}
	In order to obtain \eqref{eq:oscillation}, we will prove the following two claims. Then it suffices to repeat the final part of the proof in \cite[Lemma~16]{MR1438304}---we will  skip those details.\\
	\textbf{Claim 1.}  There exist $c = c(d,\alpha,\unD)$ and $q = q(d,\alpha,\unD)\in (0,1)$ such that for $-1\leq l \leq k-2$ and $x\in B_k$,
	\begin{align*}
		u_k^l(x) \leq C(q^{k_0})^{k-l-1}u(x),\\
		v_k^l(x) \leq C (q^{k_0})^{k-l-1}v(x).
	\end{align*}
	We define the oscillation of function $f$ as $\Osc_{A} f = \sup\nolimits_A f - \inf\nolimits_A f.$\\
	The constant $k_0$ is specified so that $q^{k_0}$ is a sufficiently small number from the interval $[1/2,1)$---we refer to \cite{MR1438304} for details. \\
	
	\noindent\textbf{Claim 2.} Let $g(x) = u_{k+1}^k(x)/v_{k+1}^k(x)$. Then there is $\delta = \delta(d,\alpha,\unD)$ such that $\Osc_{B_{k+2}} g \leq \delta \Osc_{B_k} g$.\\
	
	\noindent We will now prove Claim 1 for $u$, the proof for $v$ is identical. First let $0\leq l \leq k-2$. By Lemma~\ref{lem:gdauxest},
	\begin{align*}
		u_k^l(x) = \int_{\Pi_l} G_D(z,v) P_{B_k}(x,v)\, dv \leq cG_D(z,y) \mP^x(X_{\tau_{B_k}}\in \Pi_l).
	\end{align*}
	Furthermore, since $k\geq 1$, Lemma~\ref{lem:gdauxest} yields $G_D(z,y)\leq cG_D(z,x)$. Therefore,
	\begin{align}\label{eq:claim11}
		u_k^l(x) \leq cG_D(z,x) \mP^x(X_{\tau_{B_k}}\in \Pi_l).
	\end{align}
	Recall the explicit formula for the Poisson kernel of the ball---see, e.g., Landkof \cite{MR350027}:
	\begin{align}\label{eq:pkformula}
		P_{B(0,r)}(x,v) = c_{d,\alpha} \frac{(r^2 - |x|^2)^{\alpha/2}}{(|v|^2 - r^2)^{\alpha/2}}|x-v|^{-d},\quad x\in B(0,r),\ v\in B(0,r)^c.
	\end{align}
	Using the formula, we find that
	\begin{align*}
		\mP^x(X_{\tau_{B_k}}\in \Pi_l) = \int_{\Pi_l} P_{B_k}(x,v)\, dv &\leq c_{d,\alpha} (r(2^{k_0})^{-k})^\alpha  \int_{\Pi_l} (|v-y|^2 - (r(2^{k_0})^{-k})^2)^{-\alpha/2}|x-v|^{-d}\, dv\\
		&\leq \tilde{c}_{d,\alpha} \frac{(r(2^{k_0})^{-k})^\alpha}{(r(2^{k_0})^{-l-1})^{\alpha}} = \tilde{c}_{d,\alpha}(2^{-k_0\alpha})^{k-l-1}.
	\end{align*}
	Coming back to \eqref{eq:claim11} we get Claim 1 for $0\leq l \leq k-2$.
	
	Now, let $l=-1$. Using \eqref{eq:pkformula}, we get
	\begin{align*}
		u_k^{-1}(x) &\leq C(d,\alpha,\unD)\int_{D\setminus B(y,r)} G_D(z,v) \frac{((r(2^{k_0})^{-k})^2 - |x-y|^2)^{\alpha/2}}{(|v-y|^2 - (r(2^{k_0})^{-k})^2)^{\alpha/2}} |x-v|^{-d}\, dv\\
		&\leq C(d,\alpha,\unD)((2^{k_0})^{-k})^{\alpha}\int_{D\setminus B(y,r)} G_D(z,v) \frac{r^\alpha}{(|v-y|^2 - (r(2^{k_0})^{-k})^2)^{\alpha/2}}|x-v|^d\, dv\\
		&\leq c(d,\alpha)C(d,\alpha,\unD)(2^{-\alpha k_0})^{k}\int_{D\setminus B(y,r)} G_D(z,v) \frac{(r^2 - |x-y|^2)^{\alpha/2}}{(|v-y|^2 - r^2)^{\alpha/2}}|x-v|^d\, dv\\
		&\leq \tilde{c}(d,\alpha)C(d,\alpha,\unD)(2^{-\alpha k_0})^{k}\int_{D\setminus B(y,r)} G_D(z,v) P_{B(y,r)}(x,v)\, dv.
	\end{align*}
	Since $G_D(z,\cdot)$ is harmonic in $D\setminus \{z\}$, the last integral is equal to $G_D(z,x)$, which yields Claim 1 for $l=-1$. Thus, Claim 1 is proved.
	%				Thus it suffices to show that
	%				\begin{align}
		%					\int_{D\setminus B(y,r)} G_D(z,v) |x-v|^{-d-\alpha}\, dv \leq C(d,\alpha,\lambda,r_0) r^{-\alpha} G_D(z,x),
		%				\end{align}
	%				which we do below. To this end we split the integral as follows:
	%				\begin{align*}
		%					\int_{D\setminus B(y,r)} G_D(z,v) |x-v|^{-d-\alpha}\, dv = \int_{(D\setminus B(y,r))\cap B(z,4r)} + \sum\limits_{l=1}^{N_0}\int_{(D\setminus B(y,r))\cap (B(z,2^{l+2}r)\setminus B(z,2^{l+1}r))},
		%				\end{align*}
	%				where $N_0 = \lceil \log_2(\diam (D)/4r)\rceil$. We denote $F_0 = (D\setminus B(y,r))\cap B(z,4r)$ and $F_l = (D\setminus B(y,r))\cap (B(z,2^{l+2}r)\setminus B(z,2^{l+1}r))$ for $1\leq l\leq N_0$. We have
	%				\begin{align*}
		%					\int_{F_0}  G_D(z,v)|x-v|^{-d-\alpha}\, dv \leq 2^{d+\alpha}r^{-d-\alpha} \int_{B(z,4r)} G_D(z,v)\, dv.
		%				\end{align*}
	
	It remains to prove Claim 2, which we do now. Let $a_1 = \inf\nolimits_{B_k} g$ and $a_2 = \sup\nolimits_{B_k} g$. Without any loss of generality we may assume $a_1\ne a_2$. Then, we let
	\begin{align*}g'(x) = \frac{g(x) - a_1}{a_2 - a_1},\quad x\in B_k.\end{align*}
	We have $0\leq g'\leq 1$, $\Osc_{B_k} g' = 1$, and $\Osc_{B_{k+2}} g = \Osc_{B_{k+2}} g' \Osc_{B_k} g$.
	If $\sup\nolimits_{B_{k+2}} g' \leq \tfrac 12$, then we are done, so assume otherwise. Note that
	\begin{align*}
		g'(x) = \frac{\frac{u_{k+1}^k(x) - a_1v_{k+1}^k(x)}{a_2-a_1}}{v_{k+1}^k(x)} =: \frac{g_1(x)}{g_2(x)},\quad x\in B_{k+2}.
	\end{align*}
	By \eqref{eq:pkformula}, we have
	\begin{align}\label{eq:supinf1}
		1\leq \frac{\sup_{B_{k+2}} g_2}{\inf_{B_{k+2}} g_2}= \frac{\sup_{B_{k+2}} v_{k+1}^k}{\inf_{B_{k+2}} v_{k+1}^k} \leq C(d,\alpha).
	\end{align}
	Furthermore, since $v_{k+1}^k(x) \leq \sup\nolimits_{B_0}v\leq C(d,\alpha,\unD)$ for all $x\in \mR^d$, we get
	\begin{align*}
		g_1(x) = v_{k+1}^k(x) g'(x) \leq C(d,\alpha,\unD),\quad x\in B_k.
	\end{align*} 
	If we extend $g_1$ to be equal to 0 on $\mR^d\setminus B_k$, then $g_1$ is regular harmonic on $B_{k+1}$, nonnegative and bounded. 
 Therefore, by the Harnack inequality in an explicit scale invariant formulation \cite[Lemma~1]{MR1703823}; see also
 Bass and Levin \cite[Theorem~3.6]{MR1918242} or Grzywny \cite{MR3225805},
	\begin{align}\label{eq:supinf2}
		1\leq \frac{\sup_{B_{k+2}} g_1}{\inf_{B_{k+2}} g_1} \leq C(d,\alpha,\unD).
	\end{align}
	By \eqref{eq:supinf1} and \eqref{eq:supinf2}, we get
	\begin{align*}
		\inf\limits_{B_{k+2}} g' \geq C^{-2} \sup\limits_{B_{k+2}} g' \geq \tfrac 12 C^{-2}.
	\end{align*}
	Therefore,
	\begin{align*}
		\Osc_{B_{k+2}} g' \leq \max(\tfrac 12, 1 - \tfrac 12 C^{-2}) = 1 - \tfrac12 C^{-2},
	\end{align*}
	which ends the proof of Claim 2, and thus the lemma is proved.
\end{proof}
\begin{proof}[Proof of {\rm Theorem~\ref{lem:lpmodulus}}]
	By Lemma~\ref{lem:gdquot}, we can assume without loss of generality that $|y-y'|\leq 1/16$. 
	
	We first consider the case $2^{k_0}|y' - y|^{1/2} \geq \delta_D(y)$, with $k_0$ from Lemma~\ref{lem:gdint}), and let $Q\in \partial D$ be such that $|y-Q| = \delta_D(y)$. Note that $y,y'\in B(Q,2^{k_0+1}|y-y'|^{1/2})$---since $|y-y'|< 1$, we have $|y-y'| < |y-y'|^{1/2}$. We split the integral as follows:
	%NOTE: we need the square root, because we divide by r in BHP.
	\begin{align}\label{eq:splitgd}
		\int_D \bigg|\frac{G_D(y,z)}{G_D(x_0,y)} - \frac{G_D(y',z)}{G_D(x_0,y')}\bigg|^p\, dz  = \int_{D\cap B(Q,2^{k_0+2}|y-y'|^{1/2})} + \int_{D\setminus B(Q,2^{k_0+2}|y-y'|^{1/2})}.
	\end{align}	
	By \eqref{eq:gdquotbound},there exist $c = c(d,\alpha,\unD)$ and $\gamma = \gamma(d,\alpha,\unD)\in (0,\alpha)$ such that
	\begin{align*}
		&\int_{D\cap B(Q,2^{k_0+2}|y-y'|^{1/2})}\bigg|\frac{G_D(y,z)}{G_D(x_0,y)} - \frac{G_D(y',z)}{G_D(x_0,y')}\bigg|^p\, dz\\ &\leq 2^p\int_{D\cap B(Q,2^{k_0+2}|y-y'|^{1/2})}\bigg(\bigg|\frac{G_D(y,z)}{G_D(x_0,y)}\bigg|^p + \bigg|\frac{G_D(y',z)}{G_D(x_0,y')}\bigg|^p\bigg)\, dz\\
		&\leq c \int_{B(0,2^{k_0+2}|y-y'|^{1/2})} |z|^{p(\alpha-\gamma-d)}\, dz\\
		&= c C(d,\alpha,p) |y-y'|^{(d+p(\alpha-\gamma-d))/2}.
	\end{align*}
	In the second integral of \eqref{eq:splitgd} we use the boundary Harnack principle given in \cite[Lemma~3]{MR1704245}: we let $u(y) = G_D(y,z)$, $v(y) = G_D(x_0,y)$ and $r = 2^{k_0+1}|y-y'|^{1/2}$ there. By the Green function estimates \eqref{eq:GDest} and arguments similar to the proof of Lemma~\ref{lem:gdquot} we find that for $z\in D\cap (B(Q,2^{k+k_0+3}|y-y'|^{1/2})\setminus B(Q,2^{k+k_0+2}|y-y'|^{1/2}))$ we have $u(A_r(Q))/v(A_r(Q))\leq C(d,\alpha,\unD) (2^k|y-y'|^{1/2})^{\alpha-\gamma-d}$, for all $k\in \{0,1,\ldots,N_0\}$, where $N_0 = \lceil \log_2 (\diam (D)/2^{k_0+2}|y-y'|^{1/2})\rceil$ and we define $u/v$ to be 0 outside $D$.
	Therefore, by \cite[Lemma~3]{MR1704245}, there exist $c$ and $\sigma>0$ depending only on $d,\alpha,\unD$ such that
	\begin{align*}
		\bigg|\frac{G_D(y,z)}{G_D(x_0,y)} - \frac{G_D(y',z)}{G_D(x_0,y')}\bigg|^p \leq c (2^k|y-y'|^{1/2})^{p(\alpha-\gamma-d)} |y-y'|^{\sigma p/2}
	\end{align*}
	holds for all $z\in D\cap (B(Q,2^{k+k_0+3}|y-y'|^{1/2})\setminus B(Q,2^{k+k_0+2}|y-y'|^{1/2}))$. It follows that 
	\begin{align*}
		&\int_{D\setminus B(Q,2^{k_0+2}|y-y'|^{1/2})}\bigg|\frac{G_D(y,z)}{G_D(x_0,y)} - \frac{G_D(y',z)}{G_D(x_0,y')}\bigg|^p\, dz\\
		=&\sum\limits_{k=0}^{N_0}\int_{D\cap (B(Q,2^{k+k_0+3}|y-y'|^{1/2})\setminus B(Q,2^{k+k_0+2}|y-y'|^{1/2}))}\bigg|\frac{G_D(y,z)}{G_D(x_0,y)} - \frac{G_D(y',z)}{G_D(x_0,y')}\bigg|^p\, dz\\
		\leq &c |y-y'|^{\sigma p/2} \sum\limits_{k=0}^{N_0} (2^k|y-y'|^{1/2})^{p(\alpha-\gamma-d)} (2^k|y-y'|^{1/2})^d \\
		= &c|y-y'|^{\sigma p/2} \sum\limits_{k=0}^{N_0} (2^k|y-y'|^{1/2})^{d + p(\alpha-\gamma-d)}.
	\end{align*}
	The last sum is comparable to $\diam(D)^{d+p(\alpha-\gamma-d)}$, so the proof is complete when $2^{k_0}|y-y'|^{1/2} \geq \delta_D(y)$.
	
	Now assume that $2^{k_0}|y-y'|^{1/2}<\delta_D(y)$. We split the integral in the following way:
	\begin{equation}\label{eq:gdsplit2}
		\begin{split}
			&\int_D  \bigg|\frac{G_D(y,z)}{G_D(x_0,y)} - \frac{G_D(y',z)}{G_D(x_0,y')}\bigg|^p\, dz \\
			= &\int_{D\cap B(y,2^{k_0}|y-y'|^{1/2})} + \int_{D\cap B(y,2^{k_0}|y-y'|^{1/2})^c\cap B(Q,3\delta_D(y))^c} + \int_{D\cap B(y,2^{k_0}|y-y'|^{1/2})^c\cap B(Q,3\delta_D(y))}.
		\end{split}
	\end{equation}
	The first two integrals are handled as the ones in \eqref{eq:splitgd}. In particular, in the second one we can use the boundary Harnack principle. In the last integral on the right-hand side of \eqref{eq:gdsplit2} we will apply Lemma~\ref{lem:gdint}. To this end, we split once more:
	\begin{align*}
		\int_{D\cap B(y,2^{k_0}|y-y'|^{1/2})^c\cap B(Q,3\delta_D(y))}   \leq \sum\limits_{k=0}^{M_0}  \int_{D\cap B(Q,3\delta_D(y))\cap (B(y,2^{k+k_0+1}|y-y'|^{1/2})\setminus B(y,2^{k+k_0}|y-y'|^{1/2}))} =: \sum\limits_{k=0}^{M_0} I_k,
	\end{align*}
	where $M_0 = \lceil \log_2 (3\delta_D(y)/(2^{k_0}|y-y'|^{1/2}))\rceil$. We then use Lemma~\ref{lem:gdint} with $r = r_k = 2^{k_0+k}|y-y'|^{1/2}/4$:
	\begin{align*}
		I_k &\leq C(d,\alpha,\unD) |y-y'|^{\sigma p/2} 
 \!\!\!\!\!\!\!\!\!\!\int\limits_{ B(y,2^{k+k_0+1}|y-y'|^{1/2})\setminus B(y,2^{k+k_0}|y-y'|^{1/2})} \!\!\!\!\!\!\!\!\!\!\!\!\!\!\!\!\!\!\!\!\!\!\!\!\!r_k^{p(\alpha-d-\gamma)}\, dz\\
		&\leq \widetilde{C}(d,\alpha,\unD)  |y-y'|^{\sigma p/2} (2^{k+k_0}|y-y'|^{1/2})^{d-p(\alpha-\gamma-d)},\quad k=0,\ldots,M_0,
	\end{align*}
 since for $|y-y'|\leq 1/16$, we have $|y-y'|\leq |y-y'|^{1/2}/4$, so $y'\in B(y,2^{-k_0}r)$.
	Therefore we get
	\begin{align*}
		\sum\limits_{k=0}^{M_0} I_k \leq C(d,\alpha,\unD) |y-y'|^{\sigma p/2} \delta_D(y)^{d-p(\alpha-\gamma-d)},
	\end{align*}
	which ends the proof.
\end{proof}
\begin{proof}[Proof of {\rm Theorem~\ref{th:modulus}}]
	Fix $x\in D$ and $t\in [T_1,T_2]$. First, we investigate $\eta^{x_0}$. By the results of Section~\ref{sec:spect},
	\begin{align*}
		p_t^D(x,y) = G_D\Delta^{\alpha/2}_yp_t^D(x,y).
	\end{align*}
	Furthermore, by Corollary~\ref{cor:lptd}, the function $f(y) = \Delta^{\alpha/2}_yp_t^D(x,y)$ is bounded and the bound does not depend on $x\in D$. Therefore, by Theorem~\ref{lem:lpmodulus}, for $y,y'\in D\setminus B(x_0,r_1)$,
	\begin{align*}
		\bigg|\frac{p_t^D(x,y)}{G_D(x_0,y)} - \frac{p_t^D(x,y')}{G_D(x_0,y')}\bigg| &\leq \int_D \bigg|\frac{G_D(y,z)}{G_D(x_0,y)} - \frac{G_D(y',z)}{G_D(x_0,y')}\bigg| |f(z)|\, dz\\ &\leq \bigg\|\frac{G_D(y,\cdot)}{G_D(x_0,y)} - \frac{G_D(y',\cdot)}{G_D(x_0,y')}\bigg\|_{L^1(D)} \|f\|_{\infty} \leq C|y-y'|^{\sigma},
	\end{align*}
	where the constants $C,\sigma$ depend only on $d,\alpha,\unD,T_1,T_2,x_0,$ and $r_1$. 
	
	We now proceed to $\wtn$. Note that there exist $x_1\in D$ and $r = r(\unD)$ such that $B(x_1,2r)\subset D$. Without loss of generality, we can assume that $|y-y'| < r/4$. Then, for any fixed $y,y'$, there exists $x_2$ such that $B(x_2,r/4)\subset D$ and $y,y'\notin B(x_2,r/4)$. This means that $G_D(x_2,y),G_D(x_2,y')\leq C$, where $C\geq 1$ depends only on $d,\alpha,$ and $\unD$. We then split as follows:
	\begin{align}
			\bigg|\frac{p_t^D(x,y)}{p_{t_0}^D(x_0,y)} - \frac{p_t^D(x,y')}{p_{t_0}^D(x_0,y')}\bigg| &= 	\bigg|\frac{p_t^D(x,y)}{G_D(x_2,y)} \frac{G_D(x_2,y)}{p_{t_0}^D(x_0,y)}- \frac{p_t^D(x,y')}{G_D(x_2,y')}\frac{G_D(x_2,y')}{p_{t_0}^D(x_0,y')}\bigg|\nonumber \\
			&\leq \frac{p_t^D(x,y)}{G_D(x_2,y)} \bigg|\frac{G_D(x_2,y)}{p_{t_0}^D(x_0,y)}-\frac{G_D(x_2,y')}{p_{t_0}^D(x_0,y')}\bigg| +  \frac{G_D(x_2,y')}{p_{t_0}^D(x_0,y')}\bigg|\frac{p_t^D(x,y')}{G_D(x_2,y')} - \frac{p_t^D(x,y)}{G_D(x_2,y)}\bigg|.\label{eq:split}
	\end{align}
	By using Lemma~\ref{L38} and \eqref{factorization}, we find that
	\begin{align}\label{eq:upbound}
		\frac{p_t^D(x,y)}{G_D(x_2,y)} \lesssim \frac{\mP^y(\tau_D > t)}{G_D(x_2,y)} = \frac{G_DP_t^D\kappa_D(y)}{G_D(x_2,y)} \leq C(d,\alpha,\unD,T_1,T_2) < \infty.  
	\end{align}
	By similar arguments, 
	\begin{align}\label{eq:lowbound}
			\frac{p_t^D(x_0,y)}{G_D(x_2,y)} \geq c(d,\alpha,\unD,T_1,T_2,x_0)> 0.
	\end{align}
	From \eqref{eq:split}, \eqref{eq:upbound}, \eqref{eq:lowbound}, and the H\"older regularity of $\eta^{x_0}$ obtained above, we arrive at
	\begin{align*}
		\bigg|\frac{p_t^D(x,y)}{p_{t_0}^D(x_0,y)} - \frac{p_t^D(x,y')}{p_{t_0}^D(x_0,y')}\bigg|  \leq C|y-y'|^\sigma,
	\end{align*}
	with $C$ and $\sigma$ depending on $d,\alpha,\unD,T_1,T_2,x_0$.
	
	The arguments for $\eta$ are similar to the ones for $\wtn$, with no dependence on $x_0$. The proof is complete.
\end{proof}
\section{Space-time stable processes and caloric functions}\label{sec:caloric}
\subsection{Preliminaries}
Recall that $(X_s)_{s\geq 0}$ is the isotropic $\alpha$-stable L\'evy process. Like for the space-time Brownian motion \cite{MR1814344}, we define the \textit{space-time $\alpha$-stable process} as the following L\'evy process on $\mR^{d+1}$:
$$\dot{X}_s := (-s,X_s),\quad s\geq 0.$$
Since $\dot{X}$ is a L\'evy process, it has the strong Markov property. 
Many properties of  the space-time process are inherited from the $\alpha$-stable process. 
Thus, for a (Borel) set $A\subseteq \mR^{d+1}$, we let
$$\mP^{(t,x)}(\dot{X}_s \in A) := \mP((t-s,X_s+x)\in A),$$
and for a (Borel) function $f\colon \mR^{d+1}\to \mR^d$, we have
$$\mE^{(t,x)}[f(\dot{X}_s)] = \mE[f(t-s,X_s+x)].$$
It can be easily verified that the transition probability of $\dot{X}$ takes on the following form
$$\widetilde{p}_s(t,x,du,dy) = p_s(x,y)\, dy\otimes \delta_{\{t-s\}}(du),\quad s\geq 0,\ (t,x),(u,y)\in \mR\times \mR^d.$$
 The corresponding semigroup will be denoted by $\widetilde{P}$. 
% Also, for the sake of the further development, we introduce a new notation for the Dirichlet heat kernel of $X$ in $D$:
% $$p^D(t,x,s,y) = p_{t-s}^D(x,y),\,\quad (t,x), (s,y)\in \mR\times \mR^d.$$
% The main argument for this change is that now we switch our focus to the potential theory of the space-time process, in which the Dirichlet heat kernel of $X$ for $D$ serves the role of the Green function of $\dot{X}$ in $D\times \mR$, see  Doob \cite[Section 2.IX.17]{MR1814344}. We will not make it precise here, but this way of thinking may help understanding the representation of the positive caloric functions below. The order of the variables reflects the fact that $p^D$ is non-symmetric in space-time. The first variable is the starting point and the second---the terminal point. Accordingly, the kernel $\eta_{t,Q}(x)$ is the analogue of the Martin kernel. We introduce the space-time notation for it as well:
% $$n(s,Q,t,x) = \eta_{t-s,Q}(x),\quad x\in D,\ -\infty<s<t<\infty.$$
\begin{lemma}\label{lem:stgen} The pointwise generator of the semigroup of the space-time $\alpha$-stable process coincides with the fractional heat operator $\Delta^{\alpha/2} - \partial_t$ for functions $u\in C^{1,2}_b([0,\infty)\times \mR^d)$.
\end{lemma}
\begin{proof} Let $u\in C^{1,2}_b([0,\infty)\times \mR^d)$. For all $(t,x)\in [0,\infty)\times \mR^d$ and $s\in (0,t)$, we have
	\begin{align}
	\frac 1s (\widetilde{P}_s u(t,x) - u(t,x)) &= \frac 1s \int_{\mR^d}\int_{[0,\infty)} (u(r,y) - u(t,x)) \widetilde{p}_{s}(t,x,dy,dr)\nonumber\\
	&= \frac 1s \int_{\mR^d} (u(t-s,y) -u(t,x)) p_{s}(x,y) \, dy\nonumber\\
	&= \frac 1s \int_{\mR^d} (u(t-s,y) - u(t-s,x))p_{s}(x,y) \, dy\label{eq:lapl}\\
	&\,+ \frac 1s (u(t-s,x) - u(t,x)).\label{eq:deriv}
	\end{align}
	Clearly, \eqref{eq:deriv} converges to $-\partial_t u(t,x)$ as $s\to 0^+$, so it suffices to show that \eqref{eq:lapl} converges to $\Delta^{\alpha/2}_x u(t,x)$. To this end, we will prove that
	\begin{align}\label{eq:gendiff}
	\frac 1s \int_{\mR^d} ((u(t,y) - u(t,x)) - (u(t-s,y) - u(t-s,x)))p_{s}(x,y) \, dy
	\end{align}
	converges to 0 as $s\to 0^+$.
	Let $\varepsilon>0$ and let $\delta > 0$ be so small that $p_s(x,B(x,\delta)^c) <\varepsilon$. Then we also have $p_{s'}(x,B(x,\delta)^c) < \varepsilon$ for $s'\in (0,s)$. By Lagrange's mean value theorem, we get
	$$\bigg|\frac 1s \int_{B(0,\delta)^c}((u(t,y) - u(t,x)) - (u(t-s,y) - u(t-s,x)))p_{s}(x,y) \, dy\bigg| < 2\varepsilon\|u\|_{C^{1,2}}.$$ 
	By Taylor's expansion, $u(t-s,x) = u(t,x) - s\partial_t u(t,x) + o(s)$ as $s\to 0$, and similarly for $y$, so
	\begin{align*}
	&\bigg|\frac 1s \int_{B(0,\delta)}((u(t,y) - u(t,x)) - (u(t-s,y) - u(t-s,x)))p_{s}(x,y) \, dy\bigg|\\
	=\,&\bigg|\int_{B(0,\delta)} (\partial_t u(t,x) - \partial_t u(t,y) + \frac{o(s)}s) p_{s}(x,y) \, dy\bigg|\\
	\leq\, &\delta \|u\|_{C^{1,2}} + o(1).
	\end{align*}
	This ends the proof.
\end{proof}
In the next result we exhibit a space-time Poisson kernel for cylindrical domains. As usual, for arbitrary (open) $G \subseteq \mR\times \mR^d$, we let
     $$\tau_{G} := \inf\{t>0: \dot{X}_t \notin G\}.$$ 
\begin{lemma}\label{lem:stexit}
	Recall that $D \subseteq \mR^d$ is Lipschitz open set and let $\dot{D} = (r,t)\times D$ for some (arbitrary) $-\infty\le r<t$. Then the distribution of $\dot{X}_{\tau_{\dot{D}}}$---the first exit place of $\dot{X}$ from $\dot{D}$---is given by the formula
	\begin{align*}
	\mP^{(t,x)} (\dot{X}_{\tau_{\dot{D}}} \in (ds, dy)) = \begin{cases}{\textrm{\bf 1}}_{[r,t)}(s)\, ds\otimes J^D(t,x,s,y)\, dy + \delta_{t-r}(ds)\otimes p_{t-s}^D(x,y) \, dy,\quad &r>-\infty,\\
    \textrm{\bf 1}_{(-\infty,t)}(s)\, ds\otimes J^D(t,x,s,y)\, dy,\quad &r=-\infty,
    \end{cases}
	\end{align*}
	where $$J^D(t,x,s,y) := \int_D p_{t-s}^D(x,\xi) \nu(\xi,y)\, d\xi,\quad s<t,\ x\in D,\ y\in D^c.$$
 \end{lemma}
We call $J^D$ the \textit{lateral Poisson kernel}.
 \begin{remark}
     For the cylinder $\dot{D} = (r,t)\times D$, if the process $\dot{X}$ starts at $(t,x)$ with some $x\in D$, then it immediately enters $\dot{D}$, so $\tau_{\dot{D}}>0$ almost surely, although $(t,x)\notin \dot{D}$. In the language of potential theory, the points on the \textit{top} of the cylinder are \textit{irregular}.
 \end{remark}
\begin{proof}[Proof of {\rm Lemma~\ref{lem:stexit}}.]
	Let $r>-\infty$. We have
	\begin{align}
	\mP^{(t,x)} (\dot{X}_{\tau_{\dot{D}}} \in (ds,dy)) =\ &\mP^{(t,x)} (\dot{X}_{\tau_{\dot{D}}} \in (ds,dy) ,  \tau_{\dot{D}} > \tau_D)\label{eq:dot}\\
	+\,&\mP^{(t,x)} (\dot{X}_{\tau_{\dot{D}}} \in (ds,dy) ,  \tau_{\dot{D}} = \tau_D)\label{eq:equal}\\
	+\,&\mP^{(t,x)} (\dot{X}_{\tau_{\dot{D}}} \in (ds,dy) ,  \tau_{\dot{D}} < \tau_D).\label{eq:D}
	\end{align}
	Note that \eqref{eq:dot} vanishes, because $\mP^{(t,x)}(\tau_{\dot{D}} > \tau_D) = 0$.
	
	By the Ikeda--Watanabe formula \eqref{eq:IW}, the term \eqref{eq:equal} is equal to
	\begin{align*}
	\mP^{(t,x)} (X_{\tau_D} \in A,\, \tau_{D} \leq t-r,\, \tau_D\in ds) = {\textrm{\bf 1}}_{[r,t)}(s)\, ds\otimes J^D(t,x,s,y)\, dy.
	\end{align*}
	In \eqref{eq:D} we have $\tau_D > \tau_{\dot{D}} = t-r$, so by the definition of the Dirichlet heat kernel, this term is equal to
	\begin{align*}
	\delta_{t-r}(ds)\otimes p_{t-r}^D(x,y),
	\end{align*}
see \cite[Chapter 2]{MR1329992}. 
 The case of $r=-\infty$ is left to the reader.
\end{proof}
We see that $J^D(t,x,s,y)$ represents the scenario of $\dot{X}$ starting at $(t,x)$ and leaving to $(s,y)$, where, recall, $x\in D$, $y\in D^c$, and $s<t$. 
Another way to express the result in Lemma~\ref{lem:stexit}, is as follows:
\begin{align}\label{eq:repr}
\mE^{(t,x)}u(\dot{X}_{\tau_{\dot{D}}}) = \int_r^t\int_{D^c} J^D(t,x, s,z)u(s,z) \, dz \, ds
+ \int_{D} p_{t-r}^D(x,y)u(r,y) \, dy,
\end{align}
whenever this integral makes sense, e.g., for nonnegative $u$.
By analogy to the elliptic equations, we call the right-hand side of  \eqref{eq:repr} the \textit{Poisson integral}, and the first term on the right-hand side of \eqref{eq:repr}---the \textit{lateral Poisson integral}.
\begin{remark}
	Another motivation for calling $J^D(t,x,s,z)$ the lateral Poisson kernel comes from the fact that it is the \textit{nonlocal normal derivative} of $p_{t-s}^D$, whereas $p_{t-s}^D$ serves as the Green function for the fractional heat equation. Indeed, using the definition of the nonlocal normal derivative from \cite{MR3651008}:
	\begin{equation*}
	[\partial_{\vec{n}} u](x) := \int_{D} (u(y) - u(x))\nu(x,y)\, dy,\quad x\in D^c,
	\end{equation*}
	we see that for every $z\in D^c$, 
	\begin{equation*}
	\partial_{\vec{n}} p_{t-s}^D(x,\cdot)(z) = \int_D p_{t-s}^D(x,y) \nu(y,z) \, dy = J^D(t,x,s,z), \quad x\in D.
	\end{equation*}
\end{remark}
\subsection{Caloric functions}
 We define the caloric functions in terms of the mean value property. We stress that we only consider finite nonnegative functions.
\begin{definition}
	Let $-\infty<T_1<T_2<\infty$. We say that $u\colon (T_1,T_2)\times \mR^d\to [0,\infty)$ is \textit{caloric} in $(T_1,T_2)\times D$, if the \textit{mean value property}:
	\begin{equation}\label{eq:para}
	u(t,x) = \mE^{(t,x)} u(\dot{X}_{\tau_{G}}),\qquad (t,x)\in (T_1,T_2)\times D,
	\end{equation}
	holds for every open set $G\subset\subset (T_1,T_2)\times D$.
 %,  that is, whenever  $\overline{G}\subset (T_1,T_2)\times D$. 
	
 We say that $u\colon [T_1,T_2)\times \mR^d\to [0,\infty)$ is \textit{caloric in $[T_1,T_2)\times D$} if \eqref{eq:para} holds for every open $G\subset\subset [T_1,T_2)\times D$. 
 If $u$ is caloric in $[T_1,T_2)\times D$ and satisfies \eqref{eq:para} for $G=(T_1,T_2)\times D$, then we say that $u$ is \textit{regular caloric}.
	If $u$ is caloric in $[T_1,T_2)\times D$ and $u\equiv 0$ on the \textit{parabolic boundary} $$D^p := (\{T_1\}\times D)\cup ((T_1,T_2)\times D^c),$$ then we say that $u$ is \textit{singular caloric}.
 % Let $-\infty<T_1<T_2<\infty$. We say that $u\colon (T_1,T_2)\times \mR^d\to [0,\infty)$ is \textit{caloric} in $(T_1,T_2)\times D$, if for every $(t,x)\in (T_1,T_2)\times D$,
	% \begin{equation}\label{eq:para}
	% u(t,x) = \mE^{(t,x)} u(\dot{X}_{\tau_{G}}) <\infty,
	% \end{equation}
	% for every open set $G\subset\subset (T_1,T_2)\times D$.
 % %,  that is, whenever  $\overline{G}\subset (T_1,T_2)\times D$. 
	% If \eqref{eq:para} holds for $G=(T_1,T_2)\times D$, then we say that $u$ is \textit{regular caloric} in $(T_1,T_2)\times D$.
 %  We say that $u\colon [T_1,T_2)\times \mR^d\to [0,\infty)$ is \textit{caloric in $[T_1,T_2)\times D$} if \eqref{eq:para} holds for every open $G\subset\subset [T_1,T_2)\times D$. 
	% If $u$ is caloric in $[T_1,T_2)\times D$ and $u\equiv 0$ on the \textit{parabolic boundary} $$D^p := (\{T_1\}\times \overline{D})\cup ((T_1,T_2)\times D^c),$$ then we say that $u$ is \textit{singular caloric}. 
\end{definition}
\begin{remark}\

\begin{enumerate}[label=(\alph*)]
\item Our caloric functions are just harmonic functions  of the space-time isotropic stable L\'evy process.
%We will refer to \eqref{eq:para} as the mean value property of $u$ for $G$.

\item We may also consider $T_1=-\infty$ or $T_2 = \infty$, where appropriate, in particular when defining functions caloric on $(T_1,T_2)\times D$.

\item 
The condition $G\subset\subset [T_1,T_2)\times D$ allows $G$
to \textit{touch}  $\{T_1\}\times D$.
 Caloricity in $[T_1,T_2)\times D$ may be considered as a (new) relaxation of regular caloricity,  \textit{localized} near the part $\{T_1\}\times D$ of the boundary of $(T_1,T_2)\times D$, see also Lemma~\ref{lem:regular}. Both notions are meant to facilitate discussion of boundary conditions (they generalize to harmonic functions of other strong Markov processes).
\item The caloricity in $[T_1,T_2)\times D$ helps to handle initial conditions which are functions, but also rules out some interesting cases, e.g., $(t,y)\mapsto p_t^D(x,y)$. See also \cite{MR2365478}.
%on $[0,\infty)\times D$.
Remarkably, every (nonnegative) function caloric in $(T_1,T_2)\times D$ has a certain initial condition which is a measure, see Section~\ref{sec:repr}. 

%, since we assume that $u$ is a function. 

\item A caloric function need not satisfy the fractional heat equation pointwise, due to lack of time regularity. This can be seen using the counterexample given  by Chang-Lara and D\'avila \cite[Section~2.4.1]{MR3115838} for viscosity solutions. See also Remark~\ref{r.nL} below.
\end{enumerate}
\end{remark}
\begin{lemma}\label{lem:regular}
Regular caloricity
%in $(T_1,T_2)\times D$ 
implies caloricity in $[T_1,T_2)\times D$, which in turn implies caloricity in $(T_1,T_2)\times D$.
%	If $u$ is regular caloric in $(T_1,T_2)\times D$ then it is caloric therein. 
 Furthermore, 
%the mean value property 
\eqref{eq:para} only needs to be verified for 
%(relatively compact) open 
cylinders $G$.
 %$G = (s_1,s_2)\times U \subset\subset [T_1,T_2)\times D$ (or $(T_1,T_2)\times D$).
 \end{lemma}
\begin{proof}
	Assume that \eqref{eq:para} holds for $G$. By the strong Markov property of $\dot{X}$, \eqref{eq:para} then holds for every open $G' \subset G$:
	$$u(t,x) = \mE^{(t,x)} u(\dot{X}_{\tau_{G}}) = \mE^{(t,x)} \mE^{\dot{X}_{\tau_{G'}}} u(\dot{X}_{\tau_G}) = \mE^{(t,x)} u(\dot{X}_{\tau_{G'}}).$$
	This first two assertions follow immediately. To clarify the third one, note that every open $G'\subset\subset [T_1,T_2)\times D$ is contained in an open cylinder, relatively compact in $[T_1,T_2)\times D$. Similarly for $(T_1,T_2)\times D$. %$(T_1,s_2)\times U$, with some $s_2\in (T_1,T_2)$ and $U\subset\subset D$, etc.
 %\subset\subset (T_1,T_2)\times D$ and also use the strong Markov property. 
 %Similarly for $G'\subset\subset [T_1,T_2)\times D$.
\end{proof}
We continue with several examples of caloric functions.
\begin{example}\label{ex:hkpara}
	For every fixed $x\in \mR^d$, the function $(t,y)\mapsto p_t^D(x,y)$ satisfies the mean value property on every $(\eps,T)\times D$ for $0<\eps<T<\infty$, hence it is caloric on $(0,\infty)\times D$.
\begin{example}
    If we let 
    \begin{align}\label{eq:etaneg}\eta_{t,Q}(x) := 0,\quad (t,x)\in (-\infty,0]\times \mR^d \cup (0,\infty)\times D^c,\ Q\in \partial D,
    \end{align}
    then for every fixed $Q\in \partial D$, the function $(t,x)\mapsto \eta_{t,Q}(x)$ is caloric in $(-\infty,\infty)\times D$. Indeed, the mean value property in $(\eps,T)\times D$, with $0<\eps<T<\infty$ is a consequence of \eqref{Entrance}. Then, by Lemma~\ref{lem:nonsingular}, 
    \begin{align*}
        \eta_{t,Q}(x) &= \int_0^t \int_{U^c} J^D(t,x,s,z) \eta_{s,Q}(z)\, dz\, ds\\
        &= \int_{-R}^t \int_{U^c} J^D(t,x,s,z) \eta_{s,Q}(z)\, dz\, ds,
    \end{align*}
    for any $R\geq 0$.
\end{example}
 %whereas for $\eta_{t,Q}$ it is a consequence of \eqref{Entrance}.
%\todo[inline]{Singular caloric seems to depend on whether we close the time interval or not, see  Lemma~\ref{lem:Yaglompot}.} 
%With Lemma~\ref{lem:regular} at hand, the case of $p^D$ is clear. Let $U\subset\subset D$. We have
%\begin{align*}
%\int_0^t \int_{D\setminus U} \eta_{s,Q}(z)\int_U p_{t-s}^U(x,y)\nu(y,z)\, dy\, dz\, ds = \int_0^t \int_{D\setminus U} \lim\limits_{v\to Q} \frac{p_s^D(z,v)}{\mP^v(\tau_D>1)}\int_U p_{t-s}^U(x,y)\nu(y,z)\, dy\, dz\, ds.
%\end{align*}
%Because of the mean value property of $p_t^D$ it suffices to show that we can interchange the limit and the integral. To this end let $U\subset\subset U' \subset\subset D$ and note that on $(U'\setminus U)\times [0,t]$ the function $\frac{p_s^D(z,v)}{\mP^v(\tau_D>1)}$ is bounded, whereas on $(D\setminus U')\times[0,t]$ the functions $(\tfrac{p_s^D(z,v)}{\mP^v(\tau_D>1)})_{v\in D}$ are uniformly integrable (see \eqref{factorization}) and $\int_U p_{t-s}^U(x,y)\nu(y,z)\, dy$ is uniformly bounded. Therefore, by Vitali's convergence theorem the limit and the integral can be interchanged, which ends the proof.
\end{example}

\begin{example}
	If $f\colon \mR^d \to [0,\infty)$ is a nonnegative measurable function and $P_{1}^D f(x)$ is finite for all $x\in D$, then $(t,x)\mapsto P_t^D f(x)$ is  caloric in $[0,\infty)\times D$, with the usual convention $P_0^Df := f$.
\end{example}
The following class of functions is of particular interest for us. We will show in the next section that it coincides with the class of all singular caloric functions.
\begin{lemma}\label{lem:Yaglompot}
	If $\mu(dQ\, ds)$ is a locally finite nonnegative Borel measure on $\partial D\times [0,\infty)$, then  
	\begin{align}\label{eq:singularcal}
		h(t,x) := \begin{cases}\int_{[0,t)} \int_{\partial D} \eta_{t-\tau,Q}(x)\mu(dQ\, d\tau),\quad &t>0,\ x\in D,\\
		0,\quad & \mbox{ elsewhere},\end{cases}
	\end{align} 
	is singular caloric in $[0,\infty)\times D$.
\end{lemma}
\begin{proof}
    By Lemma~\ref{lem:nonsingular}, $h$ is finite for all $t>0$ and $x\in D$, and by \eqref{eq:etaneg}, we have
    \begin{align*}
    \int_{[0,t)} \int_{\partial D} \eta_{t-\tau,Q}(x)\mu(dQ\, d\tau) = \int_{[0,\infty)} \int_{\partial D} \eta_{t-\tau,Q}(x)\mu(dQ\, d\tau),\quad t\geq 0,\ x\in D.
    \end{align*}
    Therefore, the mean value property for $h$ follows from Fubini--Tonelli and caloricity of $\eta$.
\end{proof}
\begin{remark}\label{r.nL}
    We note that the viscosity solution considered in \cite[Section~2.4.1]{MR3115838}, although non-differentiable, is Lipschitz in time. The function $n_{t,Q}$ is not even Lipschitz in $t$ because for $t\in (0,1)$ and fixed $x\in D$,
    \begin{align*}
        \frac{\eta_{t,Q}(x)}{t} &= \frac 1t \lim\limits_{y\to Q}\frac{p_t^D(x,y)}{\mP^y(\tau_D>t)}\frac{\mP^y(\tau_D>t)}{\mP^y(\tau_D>1)}
        \gtrsim \frac{p_t(x,Q)}{t} \lim_{y\to Q} \frac{\mP^y(\tau_D>t)}
    {\mP^y(\tau_D>1)} \gtrsim |x-Q|^{-d-\alpha}\lim_{y\to Q} \frac{\mP^y(\tau_D>t)}
    {\mP^y(\tau_D>1)}.
    \end{align*}
    We see, indeed, that the last limit is comparable to $t^{-1/2}$ if $D$ is $C^{1,1}$ by \eqref{eq:CKS}. Furthermore, for Lipschitz $D$ it also explodes as $t\to 0^+$ because of the proof of Lemma~\ref{lem:survmonot} and \cite[Lemma~3]{MR1438304}.
\end{remark}
\begin{lemma}\label{lem:l1loc}
	If $u$ is caloric in $(T_1,T_2)\times D$ for some $T_1<T_2$, then $u\in L^1_{\rm loc}((T_1,T_2)\times \mR^d)$.
\end{lemma}
\begin{proof}
The proof is similar to the one of \cite[Lemma 4.5]{MR4088505}. First note that for any fixed $x\in D$, $r>0$, and $B=B(x,r)$, by \eqref{factorization} we have
\begin{align*}
J^{B}(t,x, s,z) = \int_{B} p_{t-s}^{B}(x,y)\nu(y,z)\, dy &\approx \int_B p_{t-s}(x,y)\mP^y(\tau_B > t-s)\nu(y,z)\, dy\\&\geq c\int_{B(x,r/2)} p_{t-s}(x,y)\, dy \geq C>0,
\end{align*}
with $C$ depending only on $r$ and $R$, where $\delta_B(z),t-s\leq R$. Thus, $J^B(t,x,\cdot,\cdot)$ is locally bounded from below. Now, take two disjoint balls $B_1,B_2\subseteq D$, centered at some points $x_1,x_2\in D$ respectively, and let $T_1<t_0<t<T_2$ and $R>0$. Since $u$ is nonnegative and caloric, for $i=1,2$ we get
\begin{align*}
\infty> u(t,x) \geq \int_{t_0}^t\int_{B_i^c} u(s,z) J^{B_i}(t,x,s,z)\, dz\, ds \geq C\int_{t_0}^t\int_{B(0,R)\setminus B_i} u(s,z) \, dz\, ds.
\end{align*}
Therefore $u\in L^1((t_0,t)\times (B(0,R)\setminus B_i))$ for $i=1,2$. But $B_1\cap B_2 = \emptyset$, so $u\in L^1((t_0,t)\times B(0,R))$. Since $R$ can be chosen arbitrarily large, the proof is complete.

\end{proof}
The following result shows that the so-called ancient solutions, i.e., functions caloric in a time interval of the form $(-\infty,T)$, can be conveniently studied by considering only the lateral Poisson integrals.
\begin{lemma}\label{lem:ancient}
	If $u$ is caloric in $(-\infty,T)\times D$ for some $T\in \mR$, then for all $x\in U\subset\subset D$ and $t<T$ we have 
	\begin{equation}\label{eq:ancient}
	u(t,x) = \mE^{(t,x)}[u(\tau_{(-\infty,t)\times U},X_{\tau_{(-\infty,t)\times U}})] = \int_{-\infty}^t\int_{U^c} J^U(t,x,s,z)u(s,z)\, dz \, ds.\end{equation}
	In particular, the integral on the right-hand side of \eqref{eq:ancient} is finite. 
\end{lemma}
\begin{proof}
	Let $t,x,U$ be as in the statement. By the definition of caloricity, for $v<t$ we have
	\begin{align*}
	u(t,x) = \int_v^t\int_{U^c} J^U(t,x, s,z)u(s,z)\, dz\, ds + \int_U p_{t-v}^U(x,y) u(v,y)\, dy.
	\end{align*}
	The first integral on the right-hand side increases to the right-hand side of \eqref{eq:ancient} by the monotone convergence theorem and the second integral decreases. It suffices to prove that
\begin{equation}\label{eq:ancientlim}a:=\lim\limits_{v\to-\infty} \int_U p_{t-v}^U(x,y) u(v,y)\, dy = 0.
	\end{equation}
%    We will show that if the above limit is equal to $a>0$, then the mass of $u$ in $U\times(-\infty,t)$ becomes so large that the lateral Poisson integral diverges for $(x_0,t)$ with $x_0\in D\setminus \overline{U}$. Here are the details.
To this end note that
%If the limit in \eqref{eq:ancientlim} was equal to $a>0$, then 
for every $v<t$,
\begin{align}\label{eq:contrad}\int_U p_{t-v}^U(x,y)u(v,y) \, dy \geq a.\end{align}
	Let $n >0$ be so large that $U\subset\subset D_n$ (see \eqref{eq:Dn}).  Recall that $\lambda_1(V)$ is the first eigenvalue of the Dirichlet fractional Laplacian for an open set $V$. We claim that \begin{align}\label{eq:lambdas}\lambda_1(D_n) < \lambda_1 (U).\end{align}
	A weak inequality is well known as the domain monotonicity. In order to prove the strict inequality, assume without loss of generality that $0\in U$. Then there exists $q>1$ such that $qU\subset\subset D_n$, so, by domain monotonicity,  $\lambda_1(D_n)\leq \lambda_1(qU)=q^{-\alpha} \lambda_1(U)$, which yields
 %. Furthermore, by rescaling the corresponding Rayleigh quotients it is easy to see that $\lambda_1(qU) < \lambda_1(U)$, hence 
 \eqref{eq:lambdas}.
 %is true. 
	
	By \eqref{eq:largetimes}, \eqref{eq:largetimesgen}, and the fact that each eigenfunction is bounded from above and bounded from below away from the boundary, for $s<t$, $s\to -\infty$, we get
	\begin{align*}
		\infty>u(t,x) &\geq \int_{D_n} u(s,y) p_{t-s}^{D_n}(x,y) \, dy \geq \int_U u(s,y) p_{t-s}^{D_n}(x,y)\, dy \approx  \int_{U} u(s,y)e^{-\lambda_1(D_n)(t-s)}\, dy\\
		&= e^{(-\lambda_1(D_n) + \lambda_1(U))(t-s)}\int_U u(s,y)e^{-\lambda_1(U)(t-s)}\, dy\\ &\gtrsim e^{(-\lambda_1(D_n) + \lambda_1(U))(t-s)}\int_U u(s,y) p_{t-s}^U(x,y)\, dy.
	\end{align*}
By \eqref{eq:lambdas}, we must have $a=0$.
%Therefore the assumption $a>0$ leads to a contradiction, since it as it makes the latter expression go to $\infty$ as $s\to -\infty$.
%	Let $x_0 \in D\setminus \overline{U}$ and $r = \frac 12(d(x_0,U)\wedge d(x_0,\partial D))$. We will show that the lateral Poisson integrals blow up for $(x_0,t)$ and $B(x_0,r)_{v,t}$ as $v\to-\infty$. Indeed, by using \eqref{eq:largetimes} we find that for $z\in U$ and $s<t-1$
%	$$P^{B(x_0,r)}(s,z,t,x_0) \approx \int_{B(x_0,r)}p^{B(x_0,r)}(s,y,t,x_0)|y-z|^{-d-\alpha}\, dy \approx e^{-\lambda_1(B(x_0,r))(t-s)}.$$
%	Therefore, for $v<t-1$ we have
%	$$\infty > u(x_0,t) \geq \int_v^t\int_{B(x_0,r)^c} P^{B(x_0,r)}(s,z,t,x)u(z,s)\, dz\, ds \gtrsim \int_v^{t-1}\int_U\ e^{-\lambda_1(t-s)} u(z,s)\, dz\, ds.$$
%	By using \eqref{eq:largetimes} once more, we see that, up to a multiplicative constant, this is greater than or equal to
%	$$\int_v^{t-1}\int_{U} p^U(s,y,t,x)u(y,s)\, dy\, ds \geq \int_v^{t-1} \, ds = t-1-v \mathop{\longrightarrow}\limits_{v\to-\infty} \infty,$$
%	which is the desired contradiction.
\end{proof}
\subsection{Caloric functions are continuous} This subsection is devoted to proving that caloric functions are continuous, hence locally bounded.
%, whenever the argument is separated from the lateral boundary and stays within a bounded segment of time. 
The proof is based on certain estimates for the kernel $J^D$, which may be of independent interest. Let us note in passing that \textit{bounded} caloric functions are known to be locally H\"older continuous \cite[Theorem~4.14]{MR2008600}.    %The proof uses some ideas from the elliptic case, see  \cite[Theorem 4.9]{MR4088505}, but the arguments here are more tedious, due to the work needed to estimate the kernel $J^D$.
\begin{proposition}\label{prop:cont}
	Assume that $u$ is a nonnegative caloric function in $(T_0,T_1)\times D$ for some $T_0<T_1$. Then, $u$ is continuous and 
 %consequently, 
 locally bounded therein.
\end{proposition}
We fix arbitrary $(t_0,x_0)\in (T_0,T_1)\times D$, $r\in (0,\delta_D(x_0)/2)$, and let $B_\rho = B(x_0,\rho)$ for $\rho>0$. We first establish some basic facts about the lateral Poisson kernel. With a slight conflict of notation, we introduce the Euclidean distance between sets $A,B\in \Rd$,
$$
d(A,B):=\inf\{|b-a|:\; a\in A,b\in B\}.$$
\begin{lemma}\label{lem:JDest}Let $D$ be a Lipschitz open set, $U\subset\subset D$, and $0<T<\infty$. Then,
		\begin{align}\label{eq:JDest1}
			J^D(t,x,s,z) \approx J^D(t,x_0,s,z),\quad x\in U,\ z\in D^c,\ 0<t-s<T,
		\end{align}
		and
		\begin{align}\label{eq:JDest2}
				J^D(t,x,s,z) \lesssim J^D(t',x,s,z),\quad x\in U,\ z\in D^c,\ 0<t-s\leq t'-s<T,
		\end{align}
			with the comparability constants depending only on $d,\alpha,\unD,d(U,D^c)$, and $T$.
%(here, $d$ is the Euclidean distance between sets).
\end{lemma} 
\begin{proof}
	Let $U'$ be such that $U\subset\subset U' \subset\subset D$. We pick $U'$ so that the constants below  depend only on $D$ and $U$, e.g., by assuming $d(U,D^c)/2 \geq d(U',D^c) \geq d(U,D^c)/3$. We first prove \eqref{eq:JDest1}. By \eqref{factorization},
	\begin{align}
		\nonumber J^D(t,x,s,z)& = \int_D p_{t-s}^D(x,y)\nu(y,z)\, dy\\
		&\approx \mP^x(\tau_D>t-s)\int_D p_{t-s}(x,y)\mP^y (\tau_D>t-s)\nu(y,z)\, dy\nonumber \\
		&\approx \mP^{x_0}(\tau_D>t-s)\bigg(\int_{D\setminus U'} + \int_{U'}\bigg) p_{t-s}(x,y)\mP^y (\tau_D>t-s)\nu(y,z)\, dy,\label{eq:JDaux}
	\end{align}
	with constants depending on $d,\alpha,\unD,d(U,D^c)$, and $T$. For $y\in D\setminus U'$, $|x-y| \approx |x_0-y|$, so by \eqref{DensityApprox},
	\begin{align*}
	\int_{D\setminus U'}p_{t-s}(x,y)\mP^y (\tau_D>t-s)\nu(y,z)\, dy \approx 	\int_{D\setminus U'}p_{t-s}(x_0,y)\mP^y (\tau_D>t-s)\nu(y,z)\, dy.
	\end{align*}
	For $y\in U'$, $\mP^y(\tau_D>t-s)\approx 1$ and $\nu(y,z)\approx \nu(x_0,z)$. Using this and the fact that $U\subset\subset U'$, we find that
	\begin{align*}
		\int_{U'} p_{t-s}(x,y)\mP^y (\tau_D>t-s)\nu(y,z)\, dy &\approx \nu(x_0,z)\int_{U'}p_{t-s}(x,y)\, dy \approx \nu(x_0,z)\int_{U'}p_{t-s}(x_0,y)\, dy\\
		&\approx 	\int_{U'} p_{t-s}(x_0,y)\mP^y (\tau_D>t-s)\nu(y,z)\, dy.
	\end{align*}
	Coming back to \eqref{eq:JDaux}, we obtain \eqref{eq:JDest1}. We now proceed to proving \eqref{eq:JDest2}. We split in a similar way:
	\begin{align*}
		J^D(t,x,s,z) = \bigg(\int_{U'} + \int_{D\setminus U'}\bigg) p_{t-s}^D(x,y)\nu(y,z)\, dy.
	\end{align*}
	By Lemma~\ref{lem:DHKmonot}, 
	\begin{align*}
		\int_{D\setminus U'} p_{t-s}^D(x,y)\nu(y,z)\, dy \lesssim 	\int_{D\setminus U'} p_{t'-s}^D(x,y)\nu(y,z)\, dy
	\end{align*}
	For the integral over $U'$ we use:\eqref{factorization}
	\begin{align*}
		\int_{U'} p_{t-s}^D(x,y)\nu(y,z)\, dy \approx \nu(x_0,z) \int_{U'}p_{t-s}^D(x,y)\, dy  \approx  \nu(x_0,z) \int_{U'}p_{t-s}(x,y) \mP^x(\tau_D>t-s)\mP^y(\tau_D>t-s)\, dy.
	\end{align*}
	For $w\in U'$ and $0<t-s<T$, we have $\mP^w(\tau_D > t-s)\approx 1$ and by \eqref{DensityApprox}, $\int_{U'}p_{t-s}(x,y)\, dy \approx 1$, with comparability constants depending only on $T,U'$, and $\unD$. It follows that 
	\begin{align*}
		&\nu(x_0,z) \int_{U'}p_{t-s}(x,y) \mP^x(\tau_D>t-s)\mP^y(\tau_D>t-s)\, dy\\ \approx &\,\nu(x_0,z) \int_{U'}p_{t'-s}(x,y) \mP^x(\tau_D>t'-s)\mP^y(\tau_D>t'-s)\, dy\\ \approx  &\int_{U'} p_{t'-s}^D(x,y)\nu(y,z)\, dy,
	\end{align*}
	which ends the proof.
\end{proof}
\begin{proof}[Proof of Proposition~\ref{prop:cont}] We will show continuity at the fixed point $(t_0,x_0)$. Let $x\in B_{r/2}$, $t_1\in (T_0,t_0)$ and $t\in (t_1,T_1)$, so that $T_1<t_1<t<T_0$. We have
\begin{align*}
	 u(t,x) &= \int_{B_r} u(t_1,y)p_{t-t_1}^{B_r}(x,y)\, dy + \int_{t_1}^t\int_{B_r} u(\tau,z) J^{B_r}(t,x,\tau,z)\, dz\, d\tau,\\
     u(t_0,x_0) &= \int_{B_r} u(t_1,y)p_{t_0-t_1}^{B_r}(x,y)\, dy + \int_{t_1}^{t_0}\int_{B_r} u(\tau,z) J^{B_r}(t_0,x,\tau,z)\, dz\, d\tau.
\end{align*}
Since $u$ is nonnegative and caloric, all integrals above are finite. For $(t,x)$ sufficiently close to $(t_0,x_0)$, we have $p_{t-t_1}^{B_r}(x,y) \approx p_{t_0-t_1}^{B_r}(x_0,y)$ uniformly in $y$. Therefore, by the dominated convergence theorem,
\begin{align*}
	\int_{B_r} u(t_1,y)p_{t-t_1}^{B_r}(x,y)\, dy \mathop{\longrightarrow}\limits_{(t,x)\to (t_0,x_0)} 	\int_{B_r} u(t_1,y)p_{t_0-t_1}^{B_r}(x_0,y)\, dy.
\end{align*}
Therefore it remains to show that 
\begin{align}\label{eq:lateralpart}
	\int_{t_1}^t\int_{B_r} u(\tau,z) J^{B_r}(t,x,\tau,z)\, dz\, d\tau \mathop{\longrightarrow}\limits_{(t,x)\to (t_0,x_0)} 	\int_{t_1}^{t_0}\int_{B_r} u(\tau,z) J^{B_r}(t_0,x_0,\tau,z)\, dz\, d\tau.
\end{align}
Assume that $t>t_0$ (we skip the other case, as it is similar). Then,
\begin{align*}
	&\bigg|\int_{t_1}^t\int_{B_r} u(\tau,z) J^{B_r}(t,x,\tau,z)\, dz\, d\tau -	\int_{t_1}^{t_0}\int_{B_r} u(\tau,z) J^{B_r}(t_0,x_0,\tau,z)\, dz\, d\tau\bigg|\\
	\leq& \int_{t_1}^{t_0}\int_{B_r} u(\tau,z) |J^{B_r}(t,x,\tau,z) - J^{B_r}(t_0,x_0,\tau,z)|\, dz\, d\tau + \int_{t_0}^t\int_{B_r} u(\tau,z) J^{B_r}(t,x,\tau,z)\, dz\, d\tau=: I_1 + I_2.
\end{align*}
By Lemma~\ref{lem:JDest}, we have $J^{B_r}(t,x,\tau,z) \lesssim J^{B_r}(t_0+\eps,x_0,\tau,z)$ for $t_1\leq \tau\leq t\leq t_0+\eps$, $x\in B_{r/2}$, and $z\in B_r^c$. Therefore by the dominated convergence theorem, $I_2 \to 0$. Furthermore, by the properties of $p_t^{B_r}$ and the dominated convergence theorem, it is easy to see that $J^{B_r}(\cdot,\cdot,\tau,z)$ is continuous on $(\tau,\infty)\times B_r$ for all $\tau\in \mR$ and $z\in D^c$. Therefore, using the bounds of Lemma~\ref{lem:JDest} and the dominated convergence theorem once again, we find that $I_1\to 0$ as well. This ends the proof.
\end{proof}
\section{Representation of caloric functions in Lipschitz open sets}\label{sec:repr}
We first discuss the representation for functions caloric on $[0,T)\times D$, where the meaning of the initial condition is clearer. We then use this case to resolve the situation of functions caloric in $(0,T)\times D$.
	\subsection{Functions caloric up to time $0$}
	\begin{lemma}\label{lem:decomp}
		Assume that $u$ is a nonnegative caloric function in $\dot{D}:=[0,T)\times D$. Then there exists a unique decomposition $u = r + s$, where $r$ is regular caloric in $\dot{D}$ and $s$ in singular caloric in $\dot{D}$.
	\end{lemma}
	\begin{proof}
		Let $t<T$. Since $u$ has the mean value property in every $\dot{D}_n= (0,t)\times D_n$ (see \eqref{eq:Dn}), we have
		\begin{align*}
		u(t,x) = \mE^{(t,x)} u(\dot{X}_{\tau_{\dot{D}_n}}) =: i_n(t,x) + l_n(t,x) + s_n(t,x),
		\end{align*}
		where
		\begin{align*}
		&i_n(t,x) = \mE^{(t,x)}[u(\dot{X}_{\tau_{\dot{D}_n}})\semicol  \tau_{D_n} > t],\\
		&l_n(t,x) = \mE^{(t,x)}[u(\dot{X}_{\tau_{\dot{D}_n}})\semicol  \tau_{D_n} < t,\, \tau_{D_n} = \tau_D],\\
		&s_n(t,x) = \mE^{(t,x)}[u(\dot{X}_{\tau_{\dot{D}_n}})\semicol  \tau_{D_n} < t,\, \tau_{D_n} < \tau_D].
		\end{align*}
		We let $n\to \infty$. By the monotone convergence, we get
		\begin{equation*}i_n(t,x) = \mE^{(t,x)}[u(\dot{X}_t)\semicol \tau_{D_n} > t] \nearrow \mE^{(t,x)}[u(\dot{X}_t)\semicol \tau_{D}>t] =: i(t,x),
		\end{equation*}
		and by \cite[(5.40)]{MR1438304},
		\begin{equation*}
		l_n(t,x) = \mE^{(t,x)}[u(\dot{X}_{\tau_{\dot{D}}})\semicol  \tau_D < t,\, \tau_{D_n} = \tau_D] \nearrow \mE^{(t,x)}[u(\dot{X}_{\tau_{\dot{D}}})\semicol  \tau_D < t] =: l(t,x),
		\end{equation*}
		the limits being finite because all $i_n$, $l_n$, and $s_n$ are nonnegative. So, $s_n(t,x)$ converges to some $s(t,x)$.
		Since $r(t,x):=i(t,x) + l(t,x) = \mE^{(t,x)}u(\dot{X}_{\tau_{\dot{D}}})$, $r$ is regular caloric. By inspecting the definition of $s_n$, we find that $s$ is singular caloric: indeed, if $X_t$ starts from $x\in D^c$, then the event $\tau_{D_n} < \tau_D$ has probability $0$, so $s_n(t,x)=0$ for $x\in D^c$, and if $\dot{X}$ starts from $(0,x)$, $x\in D$, then $s_n(0,x) = 0$ because $\tau_{D_n} \ge 0$.
		
		Assume that there is another decomposition $u = r' + s'$. Since $s'=s=0$ on $D^p$, we have that $r -  r' = 0$ on $D^{p}$ as well and therefore $r - r' = 0$ in $\dot{D}$, because $r-r'$ is regular caloric on $\dot{D}$.
	\end{proof}
	We next give an integral representation for the singular caloric part, with the use of the parabolic Martin kernel. We first prove the following technical result.
	\begin{lemma}\label{lem:dnd}
		Let $x\in D$ and $0<\varepsilon<T$ be fixed. Then there exists a modulus of continuity $\omega$, independent of $y$ and $t\in [\varepsilon,T]$, such that for $n$ large we have
		\begin{align}\label{eq:modulus}
		\bigg|\frac{p_t^{D_n}(x,y)}{\mP^y(\tau_{D_n}>1)} - \frac{p_t^{D}(x,y)}{\mP^y(\tau_D>1)}\bigg| \leq \omega\bigg(\frac 1n\bigg),\quad y\in D_n,\ t\in [\varepsilon,T].
		\end{align}
	\end{lemma}
\begin{proof}	
	First note that the expression on the right-hand side of \eqref{eq:modulus} converges to 0 as $n\to\infty$ for every fixed $y\in D$ (the expression is considered only when $1/n<\delta_D(y)$).  In order to get \eqref{eq:modulus} we will show that the convergence in uniform by using the Arzel\`a--Ascoli theorem. Indeed, by Theorem~\ref{th:modulus}, we find that $\overline{D_n}\ni y\mapsto p_t^{D_n}(x,y)/\mP^y(\tau_{D_n}>1)$ are uniformly H\"older continuous for $n$ large and $t\in [\eps,T]$. Furthermore, it is well-known that a H\"older continuous function in $\overline{D_n}$ can be extended to a function on $\overline{D}$ with the same H\"older regularity, see, e.g., Banach \cite[IV (7.5)]{MR0043161}. If we denote the corresponding extensions by $f_n$ then by the Arzel\`a--Ascoli theorem, we find that
\begin{align*}
\bigg|f_n(t,y) - \frac{p_t^{D}(x,y)}{\mP^y(\tau_D>1)}\bigg| \leq \omega\bigg(\frac 1n\bigg),\quad y\in D,\ t\in [\varepsilon,T].
\end{align*}
In particular, \eqref{eq:modulus} follows.
\end{proof}
	\begin{theorem}\label{th:repr}
		Assume that $u$ is singular caloric in $[0,T)\times D$. Then there exists a nonnegative Borel measure $\mu$ on $\partial D\times [0,T)$ such that
		\begin{equation}\label{eq:representation}
			u(t,x) = \int_{[0,t)} \int_{\partial D} \eta_{t-s,Q}(x) \mu(dQ\, ds),\quad x\in D,\ t\in (0,T).
		\end{equation}
	\end{theorem} 
	\begin{proof}
			Let $D_n$ be as in Lemma \ref{lem:decomp} and let $N$ be large enough, so that $x,x_0\in D_N$. Since $u$ is singular caloric, for natural $n>N$ we have
			\begin{align*}
			u(t,x) &=  \mE^{(t,x)}[u(\dot{X}_{\tau_{\dot{D}_n}})\semicol  \tau_{D_n} < t, X_{\tau_{D_n}} \in D\setminus D_n]\\
			&= \int_0^t\int_{D\setminus D_n} u(s,z)\int_D p_{t-s}^{D_n}(x,y)\nu(y,z)\, dy \, dz\, ds\\
			&= \int_0^t\int_D \frac{p_{t-s}^{D_n}(x,y)}{\mP^y(\tau_{D_n}>1)}\int_{D\setminus D_n} \mP^y(\tau_{D_n}>1)u(s,z) \nu(y,z)\, dz \, dy\, ds.
			\end{align*}
			We define $$\mu_n(dy\, ds) = \int_{D\setminus D_n} \mP^y(\tau_{D_n}>1) u(s,z) \nu(y,z) \, dz \, dy\, ds.$$ Note that by \eqref{factorization}, if we fix $\theta>0$, then we have $\mP^y(\tau_{D_n}>1)\lesssim p_{s+\theta}^{D_n}(x_0,y)$ uniformly in $s\in (0,t)$. Therefore, since $u$ is caloric, for $\theta$ sufficiently small we have
			\begin{align*}
				\int_0^t \int_{\mR^d}\mu_n(dy\, ds) \lesssim 	\int_0^t \int_{\mR^d}\ \int_{D\setminus D_n} p_{t+\theta-s}^{D_n}(x_0,y) u(s,z) \nu(y,z) \, dz \, dy\, ds \leq u(x_0,t+\theta),
			\end{align*}
			which means that the masses of $\mu_n$ are uniformly bounded. With this notation we have
			\begin{align*}
			u(t,x) = \int_0^t \int_D \frac{p_{t-s}^{D_n}(x,y)}{\mP^y(\tau_{D_n}>1)}\, \mu_n(dy\, ds).
			\end{align*} 
			The goal is then to show that the right-hand side converges to the right-hand side of \eqref{eq:representation}. To this end we will isolate small times and look separately at $D_N$ and $D\setminus D_N$.
			
			Note that all $\mu_n$ are supported in $D\times [0,T]$, so the sequence $(\mu_n)$ is tight and we can extract a subsequence $\mu_{n_k}$ converging weakly to $\mu$. Furthermore, for every $U\subset\subset D$ and $0<t<T$, we have that $\mu_n(U\times[0,t]) \to 0$ as $n\to \infty$, so $\mu|_{\overline{D}\times [0,T)}$ must be concentrated on $\partial D\times [0,T)$.

			%Since the measures $\mu_n$ are absolutely continuous, with the densities approximately constant in $n$ on $D_N\times (0,t)$, and $\mu_n(D_{N}\times[0,t]) \to 0$, the dominated convergence theorem gives
			Since for $y\in D_N$ we have $p_{t-s}^{D_n}(x,y) \approx p_{t-s}^D(x,y)$ for $n>N+1$, we find that\begin{equation}\label{eq:dno}
			\lim\limits_{n\to\infty}\int_0^t\int_{D_{N}} \frac{p_{t-s}^{D_n}(x,y)}{\mP^y(\tau_{D_n}>1)} \mu_n(dy\, ds) \lesssim  \lim\limits_{n\to\infty}\int_0^t\int_{D\setminus D_n} u(s,z)\int_{D_N} p_{t-s}^{D}(x,y)\nu(y,z)\, dy \, dz\, ds = 0.
			\end{equation}
			We will now show that there exists a modulus of continuity $\omega$ independent of $n$ such that
			\begin{equation}\label{eq:epstime}
			\int_{t-\epsilon}^t\int_{D} \frac{p_{t-s}^{D_n}(x,y)}{\mP^y(\tau_{D_n}>1)} \mu_n(dy\, ds) < \omega(\epsilon).
			\end{equation}
			To this end we will show that the left-hand side converges to 0 as $\epsilon \to 0^+$ for each $n>N$, and that it is nonincreasing with respect to $n$ for each (small) $\epsilon$. By the definition of $\mu_n$ and the fact that $u$ is caloric,
	\begin{align}\label{eq:nomassforward}
				\int_{t-\epsilon}^t\int_{D} \frac{p_{t-s}^{D_n}(x,y)}{\mP^y(\tau_{D_n}>1)} \mu_n(dy\, ds) &= 	\int_{t-\epsilon}^t\int_{D\setminus D_n} J^{D_n}(t,x,s,z)u(s,z)\, dz\, ds\\
				&= u(t,x) - \int_{D_n} p_{\epsilon}^{D_n}(x,y)u(t-\epsilon,y)\, dy.\nonumber
			\end{align}
			The last expression converges to 0 for $\eps\to 0^+$ for all fixed $n$, because $u$ is continuous in both variables, and it is nonincreasing with respect to $n$ because of the domain monotonicity. This proves \eqref{eq:epstime}.

			Note also that the right-hand side of \eqref{eq:representation} is finite because $\mu$ is a finite measure and $\eta_{s,Q}(x)$ is bounded in $s$ and $Q$ for fixed $x$. Therefore,
			\begin{align}\label{eq:limitepstime}
			\lim\limits_{\epsilon\to 0^+}\int_{[t-\epsilon,t)}\int_{\partial D} \eta_{t-s,Q}(x) \mu(dQ\, ds) = 0.
			\end{align}
			By \eqref{eq:dno},\eqref{eq:epstime}, and \eqref{eq:limitepstime}, for any $\delta>0$ there exist $\epsilon$ (small) and $N_0$ (large) such that for $n>N_0$,
			\begin{align*}
			&\bigg|\int_0^t \int_D \frac{p_{t-s}^{D_n}(x,y)}{\mP^y(\tau_{D_n}>1)}\, \mu_n(dy\, ds) - \int_{[0,t)} \int_{\partial D} \eta_{t-s,Q}(x) \mu(dQ\, ds)\bigg|\\ \leq &\,\bigg|\int_{[t-\epsilon,t)}\int_{\partial D} \eta_{t-s,Q}(x) \mu(dQ\, ds)\bigg| + \bigg|\int_{t-\epsilon}^t\int_{D\setminus D_{N}} \frac{p_{t-s}^{D_n}(x,y)}{\mP^y(\tau_{D_n}>1)} \mu_n(dy\, ds)\bigg| + \bigg|\int_0^t\int_{D_{N}} \frac{p_{t-s}^{D_n}(x,y)}{\mP^y(\tau_{D_n}>1)} \mu_n(dy\, ds)\bigg|\\
			+&\,\bigg|\int_0^{t-\epsilon} \int_{D\setminus D_N} \frac{p_{t-s}^{D_n}(x,y)}{\mP^y(\tau_{D_n}>1)}\, \mu_n(dy\, ds) - \int_{[0,t-\epsilon)} \int_{\partial D} \eta_{t-s,Q}(x) \mu(dQ\, ds)\bigg|\\
			\leq &\,3\delta + \bigg|\int_0^{t-\epsilon}\int_{D\setminus D_N} \frac{p_{t-s}^{D_n}(x,y)}{\mP^y(\tau_{D_n}>1)}\, \mu_n(dy\, ds) - \int_{[0,t-\epsilon)} \int_{\partial D} \eta_{t-s,Q}(x) \mu(dQ\, ds)\bigg|.
			\end{align*}
			Furthermore, if $N_0$ is large enough, then by Lemma~\ref{lem:dnd},
			\begin{align*}
				&\bigg|\int_0^{t-\epsilon} \int_{D\setminus D_N} \frac{p_{t-s}^{D_n}(x,y)}{\mP^y(\tau_{D_n}>1)}\, \mu_n(dy\, ds) - \int_{[0,t-\epsilon]} \int_{\partial D} \eta_{t-s,Q}(x) \mu(dQ\, ds)\bigg|\\ \leq &\,\delta + 	\bigg|\int_0^{t-\epsilon} \int_{D\setminus D_N} \frac{p_{t-s}^{D}(x,y)}{\mP^y(\tau_{D}>1)}\, \mu_n(dy\, ds) - \int_{[0,t-\epsilon]} \int_{\partial D} \eta_{t-s,Q}(x) \mu(dQ\, ds)\bigg|.
		\end{align*}
			By Lemma~\ref{lem:noup}, $\mu_n\cdot \textbf{1}_{D\times [0,t-\eps]}\to \mu\textbf{1}_{D\times [0,t-\eps)}$ weakly. By Corollary~\ref{cor:MYcont}, $(s,y)\mapsto \frac{p_{t-s}^D(x,y)}{\mP^y(\tau_{D_n}>1)}$ is in $C([0,t-\epsilon]\times \overline{D})$. So, the last expression is smaller than $2\delta$ for $n$ large enough, which ends the proof.
%			Therefore it suffices to show that for every $\delta > 0$
%			\begin{equation*}
%			\int_\delta^t\int_{D} \frac{p_s^{D_n}(x,y)}{p_s^{D_n}(x_0,y)} \mu_n(dy\, ds) \mathop{\longrightarrow}\limits^{n\to\infty} \int_\delta^t \int_{\partial D} \eta_{s,Q}(x) \mu(dQ, ds).
%			\end{equation*}
	\end{proof}

 \begin{theorem}
     The measure $\mu$ obtained in Theorem~\ref{th:repr} is unique.
 \end{theorem}
 \begin{proof}
     Following \cite{MR1438304,MR2365478}, we start by showing that the measures $\mu_n^Q$ corresponding to $\eta_{t,Q}$ converge to $\delta_Q\otimes \delta_0$ for $t>0$, $Q\in \partial D$. To this end, fix $Q\in \partial D$ and let
     \begin{align*}
         \mu_n^Q(y,s) = \mP^y(\tau_{D_n}>1) \int_{D\setminus D_n} \eta_{s,Q}(z) \nu(y,z)\, dz,\quad s>0,\ y\in \mR^d.
     \end{align*}
    By Lemma~\ref{lem:nonsingular}, $\mu_n^Q((B(Q,\varepsilon)\times [0,\varepsilon))^c) \to 0$ as $n\to \infty$, for any $\varepsilon>0$. So, $\mu_n$ converges weakly to $\delta_Q\otimes \delta_0$.

    Now, let $u$ be a singular caloric function and assume that 
    \begin{align*}
        u(t,x) = \int_{[0,t)}\int_{\partial D} \eta_{t-s,Q}(x) \mu(dQ\, ds).
    \end{align*}
    Let $\mu_n(y,s) = \int_{D\setminus D_n} \mP^y(\tau_D>1)u(s,z)\nu(y,z)\, dz$. By Fubini--Tonelli,
    \begin{align*}
        \mu_n(y,s) &= \int_{D\setminus D_n} \mP^y(\tau_D>1) \nu(y,z)\int_{[0,s)}\int_{\partial D} \eta_{s-\tau,Q}(z)\mu(dQ\, d\tau)\, dz\\
        &=\int_{[0,s)}\int_{\partial D} \mu_n^Q(y,s-\tau) \mu(dQ\, d\tau).
    \end{align*}
    Let $f\in C_b( \overline{D} \times [0,T])$. Then,
    \begin{align*}
        \int_0^t \int_{D} f(y,s) \mu_n(y,s)\, dy\, ds &= \int_0^t \int_{D} f(y,s)\int_{[0,s)} \int_{\partial D} \mu_n^Q(y,s-\tau)\mu(dQ\, d\tau)\, dy\, ds\\
        &= \int_{[0,t)} \int_{\partial D} \int_0^{t-\tau}\int_{D} f(y,s+\tau) \mu_n^Q(y,s) \, dy\, ds\, \mu(dQ\, d\tau).
    \end{align*}
    Since $\mu_n^Q \implies \delta_Q\otimes \delta_0$, the above integral with respect to $dy\, ds$ converges to $f(Q,\tau)$. Therefore, by the dominated convergence theorem, 
    \begin{align*}
        \int_0^t \int_{D} f(y,s) \mu_n(y,s)\, dy\, ds \mathop{\longrightarrow}_{n\to\infty}  \int_{[0,t)} \int_{\partial D} f(Q,s) \mu(dQ\,ds),
    \end{align*}
    which means that $\mu_n\implies \mu\cdot \textbf{1}_{
    \overline{D}\times [0,t)}$. Thus, $\mu$ is uniquely determined by $u$.
 \end{proof}
 \subsection{Functions caloric on $(0,T)\times D$}
 \begin{theorem}\label{th:noinit}
     Assume that $u$ is caloric on $(0,T)\times D$ and let $g=u|_{D^c}$. Then there exist unique bounded nonnegative measures $\mu$ on $[0,T)\times \partial D$ and $\mu_0$ on $D$ such that for all $0<t<T$ and $x\in D$,
     \begin{align*}
         u(t,x) = P_t^D \mu_0(x) + \int_{[0,t)}\int_{\partial D}  \eta_{t-s,Q}(x)\, \mu(dQ\, ds) + \int_0^t\int_{D^c} g(s,z)J^D(t,x,s,z)\, dz\, ds.
     \end{align*} 
\end{theorem}
\begin{proof}
    By the results of the previous subsection, there is a nonnegative measure $\mu$ on $(0,T)\times \partial D$ such that for all $0<\eps<t<T$ and $x\in D$,
        \begin{align*}
         u(t,x) = P_{t-\eps}^D u(\eps,\cdot)(x) + \int_{[\eps,t)}\int_{\partial D}  \eta_{t-s,Q}(x)\, \mu(dQ\, ds) + \int_\eps^t\int_{D^c} g(s,z)J^D(t,x,s,z)\, dz\, ds.
     \end{align*} 
     By nonnegativity, and the monotone convergence theorem, the last two integrals increase and converge as $\eps \to 0^+$, so that
      \begin{align*}
         u(t,x) = \lim\limits_{\eps\to 0^+}P_{t-\eps}^D u(\eps,\cdot)(x) + \int_{(0,t)}\int_{\partial D}  \eta_{t-s,Q}(x)\, \mu(dQ\, ds) + \int_0^t\int_{D^c} g(s,z)J^D(t,x,s,z)\, dz\, ds,
     \end{align*} 
     where the remaining limit exists and the expression under it decreases.
     Since $p_{t-\eps}^D(x,y) \approx p_{t}^D(x,y)$ and $p_t^D(x,\cdot)\approx 1$ for any $U\subset\subset D$ we find that $u(\eps,\cdot)$ have bounded integral on $U$. Therefore, by the Prokhorov theorem, there is a sequence $(\eps_n)$ such that $u(\eps_n,\cdot)$ converge weakly  on compact subsets of $D$ to a measure $\mu_0$, locally finite on $D$. Furthermore, we have
     \begin{align*}
         P_{t-\eps}^D u(\eps,\cdot)(x) = \int_D p_{t-\eps}^D(x,y)u(\eps,y)\, dy = \int_D \frac{p_{t-\eps}^D(x,y)}{\mP^y(\tau_D>1)} \mP^y(\tau_D>1)u(\eps,y)\, dy.
     \end{align*}
     Since $\frac{p_{t-\eps}^D(x,y)}{\mP^y(\tau_D>1)} \approx \frac{p_{t}^D(x,y)}{\mP^y(\tau_D>1)} \approx 1$ we find that the functions $y\mapsto \mP^y(\tau_D>1)u(\eps,y)$ have bounded mass. By Prokhorov theorem, we can infer without loss of generality that $\mP^y(\tau_D>1)u(\eps_n,y)$ converge weakly to a finite measure $\widetilde{\mu}$ on $\overline{D}$. We have $\widetilde{\mu}(dy) = \mP^y(\tau_D>1)\mu(dy)$ on $D$. By \eqref{eq:timecont},
     \begin{align*}
         \lim\limits_{\eps\to 0^+} \int_D \frac{p_{t-\eps}^D(x,y)}{\mP^y(\tau_D>1)} \mP^y(\tau_D>1)u(\eps,y)\, dy &= \int_{\overline{D}} \frac{p_{t}^D(x,y)}{\mP^y(\tau_D>1)}\, \widetilde{\mu}(dy)\\
         =&\int_D p_{t}^D(x,y)\, \mu_0(dy) + \int_{\partial{D}} \eta_{t,Q}(x)\, \widetilde{\mu}(dQ).
     \end{align*}
     We end the proof by defining $\mu$ on $[0,T)\times D$ as $\mu\textbf{1}_{(0,T)\times \partial D} + \delta_0(dt)\otimes \widetilde{\mu}$.
\end{proof}
\appendix
\section{Weak and classical formulations for caloric functions}
The following result seems to be well-known, but we were unable to locate a proof. The arguments are very similar to the case of the Laplacian discussed by Hunt \cite{MR79377}. 
\begin{lemma}\label{lem:ptdfhe}
	For any $x\in D$ the function $(t,y) \mapsto p_t^D(x,y)$ is a classical solution to the fractional heat equation with the Dirichlet condition:
	\begin{align}\label{eq:fhe}
	\begin{cases}
	(\partial_t - \Delta^{\alpha/2}_y)p_t^D(x,y) = 0\quad &t>0,\ y\in D,\\
	p_t^D(x,y) = 0\quad &t>0,\ y\in D^c.
	\end{cases}
	\end{align}   
	It is also a weak solution in the sense that for $\phi \in C_c^\infty([0,\infty)\times \mR^d)$ and $0<t_1<t_2<\infty$ we have
	\begin{align*}
	\int_{t_1}^{t_2} \int_D (\partial_t + \Delta^{\alpha/2})\phi(t,y) p_t^D(x,y)\, dy\, dt = \int_D \phi(t_2,y) p_{t_2}^D(x,y)\, dy - \int_D \phi(t_1,y) p_{t_1}^D(x,y)\, dy.
	\end{align*}
\end{lemma}
\begin{proof}
	By definition, the exterior condition is satisfied, so it suffices to verify that $(\partial_t - \Delta^{\alpha/2})p_t^D(x,y) = 0$. To this end we will differentiate the Hunt formula. Using the subordination and Fourier inversion formulas (see, e.g., Bogdan and Jakubowski \cite[Lemma~5]{MR2283957}) it is easy to see that
 %by the result of Grzywny and Szczypkowski \cite[Theorem~5.2]{MR4308627} (take $y=0$ and note that $b_{[h_0^{-1}(1/t)]}$ therein is equal to 0, see \cite[(1.2)]{MR4308627}) we get that
 $p_t$ is smooth in $x$ for $t>0$ and $\partial^{\beta}_yp_t(x,y)$ is bounded whenever $|x-y|$ is separated from 0 for any $\beta\in \mathbb{N}_0$. Note that this is the case for $|X_{\tau_D}-y|$. Therefore, for fixed $(t,y)$, by the dominated convergence theorem we find
	\begin{align*}
	\partial^\beta_y p_t^D(x,y) &= \partial^\beta_y p_t(x,y) - \partial^\beta_y \mE^x[ p_{t-\tau_D}(X_{\tau_D},y)\semicol  \tau_D<t]\\
	&=\partial^\beta_y p_t(x,y) -  \mE^x[ \partial^\beta_yp_{t-\tau_D}(X_{\tau_D},y)\semicol  \tau_D<t].
	\end{align*}
	Furthermore,
	\begin{align}\label{eq:c2bound}
	\|p_t^D(x,\cdot)\|_{C^2(B(y,\delta_D(y)/2))} < \infty,
	\end{align}
	hence $\Delta^{\alpha/2}_y p_t^D(x,y)$ is well defined for $(t,y)\in D\times(0,\infty)$ and we have
	\begin{align*}
	\Delta^{\alpha/2}_y p_t^D(x,y) = \Delta^{\alpha/2}_y p_t(x,y) - \Delta^{\alpha/2}_y \mE^x[ p_{t-\tau_D}(X_{\tau_D},y)\semicol  \tau_D<t].
	\end{align*}
	We can also interchange $\Delta^{\alpha/2}_y$ with the expectation. The easiest way to see that is by using Fubini--Tonelli, \eqref{eq:c2bound} and the Taylor expansion in the following (symmetrized) representation of the fractional Laplacian:
	\begin{align*}
	\Delta^{\alpha/2}_y u(x) = \int_{\mR^d} (u(x+y) - 2u(x) + u(x-y))\nu(y)\, dy.
	\end{align*}
	Thus, we obtain
	\begin{align*}
	\Delta^{\alpha/2}_y p_t^D(x,y) = \Delta^{\alpha/2}_y p_t(x,y) -  \mE^x[\Delta^{\alpha/2}_y p_{t-\tau_D}(X_{\tau_D},y)\semicol  \tau_D<t].
	\end{align*}
	We now compute the time derivative. Note that $\partial_t p_t(x,y)$ exists and is equal to $\Delta^{\alpha/2}_y p_t(x,y)$, so it is bounded for $|x-y|$ separated from 0. We have
	\begin{align*}
	\partial_t p_t^D(x,y) = \partial_t p_t(x,y) - \partial_t \mE^x[ p_{t-\tau_D}(X_{\tau_D},y)\semicol  \tau_D<t],
	\end{align*}
	provided the last expression exists, which we now prove. Without loss of generality, let $h>0$. We have
	\begin{align*}
	&\frac 1h \big[\mE^x[p_{t+h-\tau_D}(X_{\tau_D},y)\semicol  \tau_D<t+h] - \mE^x[p_{t-\tau_D}(X_{\tau_D},y)\semicol  \tau_D<t]\big]\\
	=&\frac 1h \mE^x[p_{t+h-\tau_D}(X_{\tau_D},y) - p_{t-\tau_D}(X_{\tau_D},y)\semicol  \tau_D<t] + \frac 1h \mE^x[p_{t+h-\tau_D}(X_{\tau_D},y)\semicol  t\leq \tau_D<t+h].
	\end{align*}
	By the dominated convergence theorem, we get that 
	\begin{align*}
	\lim\limits_{h\to 0^+}\frac 1h \mE^x[p_{t+h-\tau_D}(X_{\tau_D},y) - p_{t-\tau_D}(X_{\tau_D},y)\semicol  \tau_D<t] =  \mE^x[ \partial_t p_{t-\tau_D}(X_{\tau_D},y)\semicol  \tau_D<t].
	\end{align*}
	Furthermore, by \eqref{DensityApprox},
	\begin{align*}
	\frac 1h \mE^x[p_{t+h-\tau_D}(X_{\tau_D},y)\semicol  t\leq \tau_D<t+h] \leq C \mP^x(\tau_D \in [t,t+h)),
	\end{align*}
	and the last expression converges to $0$ by the dominated convergence theorem. Therefore we get
	\begin{align*}
	\partial_t p_t^D(x,y) &= \partial_t p_t(x,y) -  \mE^x[ \partial_tp_{t-\tau_D}(X_{\tau_D},y)\semicol  \tau_D<t]\\ &= \Delta^{\alpha/2}_y p_t(x,y) -  \mE^x[\Delta^{\alpha/2}_y p_{t-\tau_D}(X_{\tau_D},y)\semicol  \tau_D<t] = \Delta^{\alpha/2}_y p_t^D(x,y),
	\end{align*}
	so $p_t^D$ is a classical solution to the fractional heat equation \eqref{eq:fhe}. It is also a weak solution, as follows from integration by parts and the fact that the support of the test function $\phi$ is separated from $\partial D$.
\end{proof}
\section{Almost-increasingness}
The following result is used in the proof of Lemma~\ref{lem:JDest}.
\begin{lemma}\label{lem:DHKmonot}
For open $U\subset\subset U'\subset\subset D$ and $T>0$, there is $C = C(d,\alpha,\unD,U,d(U,(U')^c),T)$ such that 
	\begin{align*}
		p_s^D(x,y) \leq Cp_t^D(x,y),\quad x\in U,\ y\in D\setminus U',\ 0<s<t<T.
	\end{align*}
\end{lemma}
The proof (given below) relies on approximate factorization of $p_t^D$ and the next lemma on the survival probability $\mP^y(\tau_D>s)$.
%weak lower scaling of 
Of course, the latter  is nonincreasing in $s\in (0,\infty)$. The following \textit{relative} upper bound is a partial converse and may be independent interest.
%Note that $s\mapsto \mP^y(\tau_D>s)$ is decreasing.
\begin{lemma}\label{lem:survmonot}
	Let $T>0$. There exists $C = C(d,\alpha,\unD,T)$ and $\sigma\in (0,1)$ such that
	\begin{align}\label{e.rup}
		 \frac{\mP^y(\tau_D>s)}{\mP^y(\tau_D > t)} &\leq C  \bigg(\frac st\bigg)^{-\sigma},\quad y\in D,\ 0<s<t<T.
	\end{align}
\end{lemma}
\begin{proof}
By \cite[Remark~3]{MR2722789} and scaling, $\mP^y(\tau_D>t_1)\approx \mP^y(\tau_D>t_2)$, uniformly in $y\in D$ and $t_1,t_2$ in each compact subset of $(0,\infty)$.
Therefore we may assume that $s$ and $t$ in \eqref{e.rup} are small. Then we can also assume that $y$ is close to the boundary, otherwise the terms on the left-hand side of \eqref{e.rup} are bounded from below by the survival probability of a sufficiently small ball (and above by $1$).
%We first assume that $T^{1/\alpha}<r_0/2$, so $D$ is $\kappa$-fat at the scales $t^{1/\alpha}$ and $s^{1/\alpha}$. 
In this setting, recalling the notation of Section~\ref{sec:geom}, by \cite[Theorem~2]{MR2722789} and \cite[Lemma~17]{MR1991120} we get 
	\begin{align}\label{e.rmr}
		\frac{\mP^y(\tau_D>s)}{\mP^y(\tau_D>t)} \approx \frac{\mE^{A_{t^{1/\alpha}}(y)}\tau_D}{\mE^{A_{s^{1/\alpha}}(y)}\tau_D}  \approx \frac{\Phi(A_{t^{1/\alpha}}(y))}{\Phi(A_{s^{1/\alpha}}(y))},
	\end{align} 
 uniformly for the considered point $y$ and times $s,t$. 
 %The comparison constants depend on $\unD$, $\alpha$, and the choice of $x_0$ and $x_1$ (we will skip such details below).
%Recall that $A_r(y)$ above are defined in  and $\Phi(y) = G_D(x_1,y)$ for a fixed reference point $x_1\in D$. 
 Let $Q\in \partial D$ be closest to $y$.
%satisfy $|y-Q| = \delta_D(y)$. 
To estimate the rightmost ratio in \eqref{e.rmr}, we consider three geometric situations:\\ 
	\textbf{Case 1.} If $y\in \mathcal{A}_{t^{1/\alpha}}(y)\cap \mathcal{A}_{s^{1/\alpha}}(y)$, then we can take $A_{t^{1/\alpha}}(y) = A_{s^{1/\alpha}}(y)=y$, proving \eqref{e.rup}\\
%and there is nothing to do.\\
	\textbf{Case 2.} If $y\in \mathcal{A}_{s^{1/\alpha}}(y)$, but  $y\notin \mathcal{A}_{t^{1/\alpha}}(y)$, then $\kappa s^{1/\alpha} \leq \delta_D(y) =|y-Q|< \kappa t^{1/\alpha}$,
so
\begin{align*}
|A_{t^{1/\alpha}}(y) - A_{t^{1/\alpha}}(Q)| & \leq |A_{t^{1/\alpha}}(y) - y| + |y-Q| + |Q- A_{t^{1/\alpha}}(Q)| 
\leq (2+\kappa)t^{1/\alpha}.
\end{align*}
By definition, \( \delta_D(A_{t^{1/\alpha}}(y))\wedge \delta_D(A_{t^{1/\alpha}}(Q)) \geq \kappa t^{1/\alpha} \).
Therefore, by the Harnack inequality \cite[Lemma~1]{MR1703823}, we find that \(\Phi(A_{t^{1/\alpha}}(y))\approx  \Phi(A_{t^{1/\alpha}}(Q))  \). 

On the other hand, since $\kappa\le 1/2$, we have $y\in \mathcal{A}_{\delta_D(y)/\kappa}(Q)$. 
In particular, we can take ${A}_{s^{1/\alpha}}(y)={A}_{\delta_D(y)/\kappa}(Q)=y$.
%Furthermore, by  By \cite[Lemma~4]{MR1438304}, $\Phi(A_{t^{1/\alpha}}(y))\lesssim \Phi(A_{t^{1/\alpha}}(Q))$. 
Then, by \cite[Lemma~5]{MR1438304}, we get
	\begin{align*}
		\frac{\Phi(A_{t^{1/\alpha}}(y))}{\Phi(A_{s^{1/\alpha}}(y))} \approx \frac{\Phi(A_{t^{1/\alpha}}(Q))}{\Phi(A_{\delta_D(y)/\kappa}(Q))} \lesssim  \bigg(\frac{t^{1/\alpha}}{\delta_D(y)/\kappa}\bigg)^{\gamma} \leq \bigg(\frac st\bigg)^{-\gamma/\alpha},
	\end{align*}
 where $\gamma = \gamma(d,\alpha,\unD)\in (0,\alpha)$. This ends the proof in this case.\\
	\textbf{Case 3.} If $y\notin \mathcal{A}_{s^{1/\alpha}}(y) \cup \mathcal{A}_{t^{1/\alpha}}(y)$, then $\delta_D(y)<\kappa s^{1/\alpha}<\kappa t^{1/\alpha}$. 
 By the same argument as in the previous case, \(\Phi(A_{t^{1/\alpha}}(y))\approx  \Phi(A_{t^{1/\alpha}}(Q))  \),
 \(\Phi(A_{s^{1/\alpha}}(y))\approx  \Phi(A_{s^{1/\alpha}}(Q))  \), and
\begin{align*}
		\frac{\Phi(A_{t^{1/\alpha}}(y))}{\Phi(A_{s^{1/\alpha}}(y))} \approx  \frac{\Phi(A_{t^{1/\alpha}}(Q))}{\Phi(A_{s^{1/\alpha}}(Q))} \lesssim    \bigg(\frac st\bigg)^{-\gamma/\alpha}
	\end{align*}
The proof is complete.
\end{proof}
Let us explain why \eqref{e.rup} is a partial converse to nonincreasingness of the survival probability. We may interpret \eqref{e.rup} as \textit{weak lower scaling} (with exponent $-\sigma$)  near $s=0$ of the survival probability $f(s):=\mP^x(\tau_D>s)$, uniform in $x\in D$. Such scaling is defined as almost-increasingness $f(s)/s^{-\sigma}\le C f(t)/t^{-\sigma}$ for (bounded arguments) $0<s\le t$; see, e.g., \cite{MR3165234}.
\begin{proof}[Proof of {\rm Lemma~\ref{lem:DHKmonot}}]
	We use \eqref{factorization}:
	\begin{align*}
		p_s^D(x,y) \approx \mP^x(\tau_D>s) p_s(x,y)\mP^y(\tau_D>s).
	\end{align*}
	Since $x\in U\subset \subset D$, we have $\mP^x(\tau_D>s) \lesssim \mP^x(\tau_D>t)$ because the latter quantity is bounded from below by a constant depending only on $U,d,\alpha,\unD$, and $T$. Furthermore, since $y\in D\setminus U'$, by \eqref{DensityApprox} we have $p_s(x,y) \approx s$. Therefore the statement of the lemma follows from Lemma~\ref{lem:survmonot}.
\end{proof}
\section{No mass concentration forward in time}
Let $\mu_n$ be the sequence of measures converging to $\mu$ constructed in the proof of Theorem~\ref{th:repr}. We will show that $\mu_n\cdot \textbf{1}_{[0,t]\times D}$ do not accumulate mass \textit{at} time $t$. Here is the precise formulation.
\begin{lemma}\label{lem:noup}
    Let $0<t<T$ and let $f\in C([0,t]\times \overline{D})$. Then
    \begin{align*}
        \lim\limits_{n\to\infty} \int_0^t\int_{D} f(s,y)\mu_n(dy\, ds) = \int_{[0,t)}\int_{\partial D} f(s,Q)\mu(dQ\, ds).
    \end{align*}
\end{lemma}
\begin{proof}
    It suffices to show that there exists a modulus of continuity $\omega$ independent of $n$ such that
    \begin{align}\label{eq:omegaeps}
         \int_{t-\eps}^t\int_{D} \mu_n(dy\, ds) < \omega(\eps).
    \end{align}
    Fix $\theta\in(0,T-t)$. By \eqref{factorization} we have
    \begin{align*}
         \int_{t-\eps}^t\int_{D} \mu_n(dy\, ds) &= \int_{t-\eps}^t\int_{D\setminus D_n}\int_{D_n}u(z,s)\mP^y(\tau_{D_n}>1)\nu(y,z)\, dy\, dz\, ds\\
         &\int_{t-\eps}^t\int_{D\setminus D_n}u(z,s)\int_{D_n}p_{t+\theta-s}^D(x,y)\nu(y,z)\, dy\, dz\, ds\\
         &\int_{t-\eps}^t\int_{D\setminus D_n}u(z,s)J^{D_n}(t+\theta,x,s,z)\, dz\, ds.
    \end{align*}
    Note that
    \begin{align*}
        &\,\int_{t-\eps}^t\int_{D\setminus D_n}u(z,s)J^{D_n}(t+\theta,x,s,z)\, dz\, ds\\
        = &\, u(t+\theta,x) - \int_{t}^{t+\theta}\int_{D\setminus D_n}u(z,s)J^{D_n}(t+\theta,x,s,z)\, dz\, ds - P_{\theta + \eps}^{D_n}u(t-\eps)(x)\\
        = &\, P_{\theta}^{D_n}u(t)(x) - P_{\theta + \eps}^{D_n}u(t-\eps)\\
        = &\,P_{\theta}^{D_n}(u(t) - P_\eps^{D_n}u(t-\eps))(x)\\
        \leq &\,\sup_{x\in D_n} (u(t,x) - P_\eps^{D_n}u(t-\eps,x)).
    \end{align*}
    The last expression decreases with $n$ for $\eps$ fixed, and converges to $0$ as $\eps\to 0^+$ for every $n>N$ because of the continuity of $u$ and strong continuity of $P_t^{D_n}$. This proves \eqref{eq:omegaeps}.
\end{proof}
\bibliographystyle{abbrv}
\bibliography{HK}

\end{document}